\documentclass{amsart}

\usepackage[cmtip,all]{xy}
\usepackage[utf8]{inputenc}
\usepackage{amsmath, amscd}
\usepackage{mathrsfs}
\usepackage{amssymb}
\usepackage{pxfonts}
\usepackage[T1]{fontenc}
\usepackage{mathtools}
\usepackage{url}
\usepackage[top=1.3in, bottom=1.3in, left=1.3in, right=1.3in]{geometry}
\usepackage{color}
\usepackage{hyperref}
\usepackage{grffile}

\DeclareMathOperator{\coker}{coker}

\newtheorem{theorem}{Theorem}[section]
\newtheorem{lemma}[theorem]{Lemma}
\newtheorem{cor}[theorem]{Corollary}
\newtheorem{prop}[theorem]{Proposition}
\theoremstyle{definition}
\newtheorem{defn}[theorem]{Definition}
\newtheorem{eg}[theorem]{Example}

\theoremstyle{remark}
\newtheorem{rmk}[theorem]{Remark}

\numberwithin{equation}{section}

\newcommand{\abs}[1]{\lvert#1\rvert}
\newcommand{\To}{\longrightarrow}
\newcommand*{\sheafhom}{\mathscr{H}\kern -.5pt om}

\begin{document}

\title[Hilbert schemes of points on singular surfaces]{The Hilbert schemes of points on surfaces with rational double point singularities}%
\author{Xudong Zheng}%
\address{Johns Hopkins University, 3400 N. Charles St. Baltimore, MD 21218, USA}%
\email{xzheng@math.jhu.edu}%

\thanks{}%
\subjclass{}%

\keywords{Hilbert scheme of points, smoothability, maximal Cohen-Macaulay modules, rational double point surface singularities}%

\date{\today}%

\begin{abstract}
We prove that the Hilbert scheme of points on a normal quasi-projective surface with at worst rational double point singularities is irreducible. 
\end{abstract}
\maketitle


The primary purpose of this article is to explore the geometry of the Hilbert scheme of points $\mathrm{Hilb}^d(X)$ for a surface $X$ with certain isolated singularities. We confirm the irreducibility of $\mathrm{Hilb}^d(X)$ for $X$ with only rational double points. 

Hilbert schemes of points on quasi-projective complex varieties, especially smooth curves and surfaces have been studied extensively in recent years. One of the common themes underlying the existing literature is that studying the Hilbert schemes $\mathrm{Hilb}^d(X)$ for varying $d$ all at once can reveal structure about the variety $X$, such as some enumerative geometric information. The smoothness of $\mathrm{Hilb}^d(X)$ often allows investigation of their further structures. On the other hand, the complexity of the singularities of $\mathrm{Hilb}^d(X)$ if $X$ is three-dimensional or higher has been one of the obstructions to generalizing results on smooth curves and surfaces. In the case that $X$ is singular, the geometry of the Hilbert schemes become more complicated. It is necessary to understand the Hilbert schemes of points of a singular variety since singular varieties are natural degenerations of smooth varieties in families. For a singular curve $C$, a naturally related object is the compactified Jacobian $J(C)$ of $C$, which admits the Abel-Jocobi map from $\mathrm{Hilb}^d(C)$, and $\mathrm{Hilb}^d(C)$ is irreducible if and only if locally any point of $C$ has embedding dimension at most 2 (locally planar singularities) (\cite{AIK, R80}). A generalized Macdonald's formula is applied in the context of the \textit{universal Hilbert schemes of points} of a family of curves leading to a proof of G\"ottsche's conjecture on the number of nodal curves on surfaces (\cite{G98, T12, S12}).  In dimension 2, if a surface $X$ has a single isolated cone singularity over a smooth curve of degree at least $5$, then $\mathrm{Hilb}^d(X)$ is reducible for sufficiently large number $d$ (\cite{MP13}). Yet it is an open question whether there is a singular surface whose Hilbert schemes of points are irreducible for any $d > 0$, particularly for surfaces with mild singularities. 

Rational double points on surfaces in some sense are among the simplest surface singularities with a complete classification. As we will see the study of reflexive modules and their deformations is essential in the proof of the main theorem, whereas we take advantage of the following facts about rational double points. First, there are finitely many isomorphism classes of indecomposable reflexive modules; and second, every module admits a free resolution over the local ring of the singularity via a matrix factorization. However, the irreducibility of the Hilbert scheme is not valid for the affine cone over a twisted cubic curve, casting doubts on generalizing our results to other quotient singularities. For a general rational surface singularity other than a rational double point, the Hilbert scheme will become reducible if $d$ is sufficiently large. We will leave the investigation of the Hilbert schemes on surfaces with minimally elliptic singularities as a separated project. 

\renewcommand*{\thetheorem}{\Alph{theorem}}
\begin{theorem}[Theorem \ref{an}]
Suppose $X$ is a quasi-projective normal surface with at worst rational double point singularities. Let $\mathrm{Hilb}^d(X)$ be the Hilbert scheme of length $d$ subschemes of $X$. Then for any positive integer $d$, $\mathrm{Hilb}^d(X)$ is irreducible of dimension $2d$.
\end{theorem}

We collect some geometric consequences of the irreducibility of $\mathrm{Hilb}^d(X)$. First of all, one can prove Fogarty's smoothness theorem for $\mathrm{Hilb}^d(X)$ on a smooth surface using the theorem of Hilbert-Burch on free resolutions of ideal sheaves of zero-dimensional subschemes (see \cite[Chap. 7]{FGA}). In our case, we show a similar smoothness criterion:

\begin{theorem}[Theorem \ref{converse}]
Suppose $Z$ be a length $d$ subscheme of the surface $X$ with only $ADE$ singularities. Then $\mathrm{Hilb}^d(X)$ is smooth at $[Z]$ if and only if $I_Z$ has finite homological dimension over $X$.
\end{theorem}

We also obtain some generalizations on results on the affine plane due to Haiman. The smooth quasi-projective variety $\mathrm{Hilb}^d(\mathbb{A}^2)$ has trivial canonical bundle, whereas:
\begin{theorem}[Theorem \ref{canonicalbundle}]
Suppose $X = \mathrm{Spec}(\mathbb{C}[x, y, z]/\langle xz - y^{n + 1}\rangle)$ is the affine surface with only $A_n$ singularity at the origin. Then $\mathrm{Hilb}^d(X)$ has a nowhere vanishing $2d$-form.
\end{theorem}

Haiman also shows that $\mathrm{Hilb}^d(\mathbb{A}^2)$ is isomorphic to the blow-up of the symmetric product $\mathbb{A}^{2(d)}$ along some ideal, where this blow-up is identified with the Hilbert-Chow morphism. We obtain analogous statement for the affine quadric cone $Q = \mathrm{Spec}(\mathbb{C}[x, y, z]/\langle xz - y^{2}\rangle)$. 
\begin{theorem}[Theorem \ref{blowup}]
The Hilbert scheme $\mathrm{Hilb}^d(Q)$ together with the Hilbert-Chow morphism is isomorphic to the blow-up of the $d$-th symmetric product $Q^{(d)}$ along the ideal consisting of products of two $S_d$-alternating functions on $Q$.
\end{theorem}

The study of deformations of closed subschemes that are supported at a closed point in a variety relies heavily on the formal analytic properties of a neighborhood of the point. We note that this is allowable for our purpose as zero-dimensional schemes are algebraic in any event and embedded deformation of closed subschemes is formal in essence.

The paper is organized as follows. In section 1, we fix terminologies on Hilbert schemes of points and review the structure of reflexive modules on $ADE$ surface singularities. Section 2 contains the results on deformations of reflexive modules, which, as syzygy modules of zero-dimensional subschemes, will give rise to deformations of the corresponding subschemes. Section 3 contains the main theorem on the irreducibility of the Hilbert scheme on surfaces with $ADE$ singularities (Theorem \ref{an}). Section 4 collects a few geometric applications of the irreducibility of $\mathrm{Hilb}^d(X)$. In particular, if $X$ is the affine quadric cone in 3-space, the Hilbert-Chow morphism has an alternative interpretation analogous to the case of the affine plane. In the last section, we provide some examples of reducible Hilbert schemes, in particular, of at least 8 points on the cone over a twisted cubic curve.

\textbf{Acknowledgement}. The results on the $A_n$ singularities case have appeared as parts of the author's Ph.D. Thesis at the University of Illinois at Chicago. It is a pleasure to acknowledge the overwhelming influence of Professors Lawrence Ein and Kevin Tucker on this work. He thanks Professors Frank-Olaf Schreyer and Robin Hartshorne for inspiring conversations contributing to the paper. He is grateful to Professors Hailong Dao, Trond Gustavsen, Akira Ishii, and Michael Wemyss for answering questions related to the paper in email correspondence. 

\section{Preliminary and motivations}

Throughout this article, we work over the field of complex numbers within the category $\mathbb{C}-\mathrm{\underline{Sch}}$ of locally Noetherian $\mathbb{C}$-schemes. Fiber products are over $\mathrm{Spec}(\mathbb{C})$ if not specified.
\renewcommand*{\thetheorem}{\thesection.\arabic{theorem}}

\subsection{Definitions and background}
Suppose $X$ is a (quasi)-projective scheme and $P(x) \in \mathbb{Q}[x]$ a numerical polynomial.

The \textit{Hilbert functor} on $X$ (with Hilbert polynomial $P$) $\underline{Hilb}^P_X$ is the contravariant functor from the category $\mathbb{C}-\mathrm{\underline{Sch}}$ to the category $\mathrm{\underline{Set}}$ of sets, which assigns to each scheme $U$ the set of closed $U$-flat subschemes $Z \subset U \times X$ that have Hilbert polynimal $P$; and for a morphism $\phi: U \to V \in \mathrm{Mor}_{\mathbb{C}-\underline{\mathrm{Sch}}}(U, V)$, $\underline{Hilb}_X(\phi)$ is the induced map of sets via fiber product. The \textit{Hilbert scheme of $X$ with Hilbert polynomial $P$}, denoted by $\mathrm{Hilb}^P(X)$, is the quasi-projective scheme which represents the functor $\underline{Hilb}^P_X$. In particular, if $P(x) = d \in \mathbb{N}$ is the constant polynomial, the scheme $\mathrm{Hilb}^d(X)$ is the \textit{Hilbert scheme of $d$ points} of $X$. There is the \textit{universal family} $\widetilde{\mathrm{Hilb}}^d(X) \xrightarrow{\pi} \mathrm{Hilb}^d(X)$.

Let $S_d$ be the symmetric group in $d$ elements, and let $X^d = X \times \dots \times X$ be the $d$-fold fiber product of $X$. The group $S_d$ acts on $X^d$ by permuting the factors. The \textit{$d$-th symmetric product} of $X$ is the projective scheme of $S_d$-quotient of $X^d$, denoted by $X^{(d)}$. Explicitly, fixing an embedding $X \hookrightarrow \mathbb{P}(\mathcal{E})$, by geometric invariant theory one can equip the set-theoretical quotient $X^d/S_d$ with a quasi-projective scheme structure. The closed points of $X^{(d)}$ are the effective 0-cycles of length $d$, hence $X^{(d)}$ is regarded as a Chow variety of 0-cycles. The symmetric product $X^{(d)}$ coincides with Rydh's \textit{scheme of divided powers} $\Gamma^d(X)$ as over $\mathbb{C}$. Over an arbitrary field $k$, for an affine scheme $Y = \mathrm{Spec}(A)$ the affine scheme of divided powers of $Y$ is the spectrum of the algebra of divided powers $\Gamma^d_k(A)$, which is not necessarily isomorphic to the $d$-th symmetric algebra $\mathrm{Sym}^d(A)$ (see \cite{Rydh08} for details).

The \textit{Hilbert-Chow morphism} is the following morphism, denoted by $h: \mathrm{Hilb}^d(X) \to X^{(d)}$. Suppose $Z \subset X$ is a length $d$ scheme supported at $r$ distinct points $z_1, \dots, z_r \in X$, giving rise to a closed point $[Z] \in \mathrm{Hilb}^d(X)$. Then $h([Z]) = \sum_{i = 1}^rm_iz_i$, where $m_i = \dim_k \mathcal{O}_{Z, z_i}$. This description of the Hilbert-Chow morphism is only on the set of closed points. The target $X^{(d)}$ should adapt the reduced scheme structure. For a precise definition see \cite[Chap. 7]{FGA}.

If $X$ is a non-singular connected curve, then $h$ is an isomorphism between two non-singular varieties. In particular, the cohomology ring of $X^{(d)}$ is expressed in terms of the cohomology ring of the curve $X$ by Macdonald's formula. If $X$ is a non-singular connected surface, then $\mathrm{Hilb}^d(X)$ is non-singular of dimension $2d$ and is a crepant resolution of singularities of $X^{(d)}$ via $h$ (\cite{F68}). 

Let $\mathcal{I} \subset \mathcal{O}_X$ be an ideal sheaf which defines a zero-dimensional subscheme $Z$ of $X$ of length $d = l(Z)$, i.e., $\dim_{\mathbb{C}}\mathcal{O}_{X}/\mathcal{I} = d$. We say that the ideal $I$ has \textit{colength} $d$. The \textit{Hilbert function} of $\mathcal{I}$ is the integral-valued vector $h = (h_0, h_1, \dots)$ where $h_i = \dim_\mathbb{C} H^0(X, \mathcal{O}_X(i)) - \dim_{\mathbb{C}} H^0 (X, \mathcal{I}(i))$ for $i \geq 0$. Localizing at a closed point $x \in X$, the \textit{punctual Hilbert scheme} $\mathrm{Hilb}^d(X, x)$ is the reduced closed subscheme of $\mathrm{Hilb}^d(X)$ parameterizing length $d$ subschemes of $X$ supported at $x$. If $x \in X$ is a non-singular point on a surface, then $\mathrm{Hilb}^d(X, x)$ is irreducible of dimension $d - 1$ (\cite{B77, I72, ES98, EL99}). 

When $X = \mathbb{P}^2$, the schemes $\mathrm{Hilb}^d(\mathbb{P}^2)$ and $\mathrm{Hilb}^d(\mathbb{P}^2, x)$ are stratified into affine spaces whose numbers of affine cells of each dimension compute the ranks of their Borel-Moore homology groups (\cite{B77, ES87, ES88, G88}). The Betti numbers of $\mathrm{Hilb}^d(X)$, Chow groups and the Chow motive with $\mathbb{Q}$-coefficients were computed for an arbitrary smooth projective surface (\cite{G90, deCM}). If $X$ is a $K3$ surface, $\mathrm{Hilb}^d(X)$ is a hyperK\"ahler manifold (\cite{B83, F83}). Write $\mathbf{H}_d \coloneqq H^*(\mathrm{Hilb}^d(X), \mathbb{Q})$ for the rational coefficient cohomology of the Hilbert scheme of points and $\mathbf{H} \coloneqq \bigoplus_{n \geq 0} \mathbf{H}_d$, then $\mathbf{H}$ is an irreducible module over a Heisenberg algebra (\cite{N97, G96}). In the case of the affine plane, the natural torus-action on $\mathbb{A}^2$ extends to $\mathrm{Hilb}^d(\mathbb{A}^2)$. The higher cohomology groups of twists of the tautological line bundles on $\mathrm{Hilb}^d(\mathbb{A}^2, 0)$ vanish. There is an Atiyah-Bott formula for the Euler characteristic of the twists of the tautological line bundle on $\mathrm{Hilb}^d(\mathbb{A}^2, 0)$ which has seen some combinatorial applications including the Macdonald positivity conjecture and the $n !$ conjecture (\cite{H98, H01}). 

\subsection{Maximal Cohen-Macaulay modules}

In this subsection we review some well-known facts about maximal Cohen-Macaulay modules over $ADE$ singularities. The standard references are \cite{GV83, AV85, E85, A86}.

\begin{defn} Let $(B, \mathfrak{m}, k)$ be a Noetherian local ring, and $M$ a finitely generated $B$-module. The $B$-module $M$ is said to be \textit{maximal Cohen-Macaulay}, or \textit{MCM}, if $M$ satisfies $\mathrm{depth}(M) = \dim B$. $M$ is said to be \textit{reflexive}, if $M \cong M^{**}$, where the functor $(\bullet)^*$ is the dual $\mathrm{Hom}_B(\bullet, B)$ on the category of $B$-modules.
\end{defn}

\begin{eg}
Let $B$ be any two-dimensional normal domain and $I \subset B$ an ideal of finite colength. Then the first syzygy module of $I$ is MCM over $B$ and is reflexive.
\end{eg}

Now let $R$ be any 2-dimensional rational double point singularity (i.e., the completion $\hat{R}$ is a two-dimensional normal local domain with singularity one of the types $A_n, D_n, E_6, E_7$ or $E_8$), then $R$ is a hypersurface ring, in particular, Gorenstein. What follows about free resolutions of finitely generated $R$-modules holds true for any \textit{abstract hypersurface} in the sense of \cite[Remark 6.2]{E80} (the maximal ideal of $R$ is minimally generated by $\dim R + 1$ elements, and the zero ideal of $R$ is analytically unmixed, or alternatively, $\hat{R}$ is the quotient of a regular local ring modulo a principal ideal. Minimal free $R$-resolution of any finitely generated $R$-module, if not finite, will be periodic of period 2 after at most 3 steps (\cite[Theorem 6.1]{E80}). A periodic resolution of period 2 is necessarily given by a \textit{matrix factorization} associated to the hypersurface. 

\begin{eg}\label{a1}
In the case of the quadric cone $R = \mathbb{C}[x, y, z]/\langle xz - y^2 \rangle$ and $Q = \mathrm{Spec}(R)$ there is only one non-free indecomposable MCM $R$-module up to isomorphism which is given by the matrix factorizations $M = M(\phi, \psi)$, where
\[\phi =
\begin{bmatrix}
x & -y \\
-y & z \\
\end{bmatrix},
\quad \textrm{ and } \quad
\psi =
\begin{bmatrix}
z & y \\
y & x \\
\end{bmatrix},
\]
and this MCM module is given by the cokernel of either one of the two matrices above. We denote the cokernel of $\phi$ by $P$. Note that $P$ has rank 1 corresponding to a Weil divisor. Geometrically, $P$ is isomorphic to the fractional ideal defining a reduced line $L$ through the singular point on $Q$, which can be understood as the first Chern class of $P$ in $\mathrm{Cl}(R)$. The projective line $\bar{L}$ as a Weil divisor in the projective closure $\bar{Q}$ is not Cartier, but $2\bar{L}$ is Cartier, i.e., $[P]$ is $2$-torsion in $\mathrm{Cl}(R)$. The minimal $R$-free resolution of $P$ is $\mathbf{F^{\bullet}}(P): \dots \to R^{\oplus 2} \xrightarrow{\phi} R^{\oplus 2} \xrightarrow{\psi} \dots \xrightarrow{\psi} R^{\oplus 2} \xrightarrow{\phi} R^{\oplus 2} \to 0$. Truncating the sequence we have a short exact sequence
\begin{equation}\label{extensionP}
0 \to P \to R^{\oplus 2} \to P \to 0.
\end{equation}
This sequence is a special case that will be studied in detail.
\end{eg}

In general, the classification of indecomposable reflexive modules over 2-dimensional rational double point singularities are fully understood (\cite[Theorem 1.11]{AV85}). Even more generally, for any two-dimensional quotient singularity given by a finite group $G \leq \mathrm{GL}(2, \mathbb{C})$ there is a one-to-one correspondence between the finite set of isomorphism classes of indecomposable reflexive modules and that of isomorphism classes of irreducible representations of the finite group $G$. In the case that $G \leq \mathrm{SL}(2, \mathbb{C})$, there is a third finite set which yields a bijection with the preceding two, namely the vertices of the dual graph of the intersection pairing of the exceptional divisor in the minimal resolution of singularity, the Dynkin diagram. This is referred as the \textit{geometric McKay correspondence}. All the indecomposable reflexive $R$-modules can be exhausted in the concrete way as follows: regarding the 2-dimensional regular ring $\mathbb{C}[u, v]$ as a regular representation of the finite group $G$ over the ring of invariants $R = \mathbb{C}[u, v]^G$, then $\mathbb{C}[u, v]$ decomposes into the direct sum of irreducible representations of $G$ over $R$ with each irreducible representation appearing exactly once, then these are precisely all of the indecomposable reflexive $R$-modules.

\begin{eg}
Explicitly in the $A_n$ singularity case, we can express $R$ as $R = \mathbb{C}[u^{n + 1}, uv, v^{n + 1}]$. Then any nontrivial indecomposable reflexive $R$-module has the form $M_i = R\langle u^av^b \mid a, b \in \mathbb{N}, a + b \equiv i (\mod n + 1)\rangle$ for $i = 1, \dots, n$. On the other hand, as an irreducible representation of $G$, each module $M_i$ as a submodule of $\mathbb{C}[u, v]$ is isomorphic to a fractional ideal in $\mathbb{C}(u, v)$, $M_i \cong R\langle 1, u^i/v^{n + 1 - i}\rangle$ for $i = 1, \dots, n$. 
\end{eg}

\section{Deformations of MCM modules}

Suppose $(X, p)$ is the germ of a rational double point. In this section, we prove the following results in the context of an arbitrary 2-dimensional rational double points: (1) given the syzygy module $M = \Omega(I_Z)$ of a zero-dimensional subscheme $Z$ on $X$, a deformation of $M$ induces a flat embedded deformation of $Z$ in $X$ (Proposition \ref{family}); (2) closed subschemes of finite projective dimension on $X$ is smoothable (Theorem \ref{finitehd}). Both statements contribute to the core of the proof of the main theorem on the irreducibility of the Hilbert schemes (Theorem \ref{an}). And (3), if the socle dimension $\mathrm{soc}(\mathcal{O}_Z)$ of a zero-dimensional subscheme $Z \subset X$ satisfies the equality $e(I_Z) - \mathrm{soc}(\mathcal{O}_Z) = 1$ then $Z$ has finite projective dimension. We also include some explicit studies of the extensions of reflexive modules. 

The first proposition motivates the entire investigation of subschemes of finite homological dimension. 

\begin{prop}\label{smoothness}
Suppose a subscheme $Z$ has finite projective dimension on the surface $X$, then the Zariski tangent space of $\mathrm{Hilb}^d(X)$ at $[Z]$ has dimension $2d$.
\begin{proof}
The proof of smoothness of the Hilbert scheme of points on a smooth surface works verbatim. We write $\mathcal{I}_Z$ for the ideal sheaf of $Z$ in $\mathcal{O}_{X}$. 
Under the assumption of the finiteness of the homological dimension of $Z$, by the Auslander-Buchsbaum formula and the theorem of Hilbert-Burch-Schaps, a minimal free resolution of $\mathcal{I}_Z$ of $Z$ takes the form:
\[
0 \to \mathcal{O}_{X}^{\oplus r} \to \mathcal{O}_{X}^{\oplus r + 1} \to \mathcal{I}_Z \to 0
\]
for some positive integer $r$. We can compute that $\mathrm{Ext}^2_{X}(\mathcal{I}_Z, \mathcal{O}_Z) = 0$ and $\mathrm{Ext}^1_{X}(\mathcal{I}_Z, \mathcal{O}_Z) \cong \mathrm{Ext}^1_{X}(\mathcal{I}_Z, \mathcal{O}_{X}) \cong \omega_Z$, which imply that $\chi(\mathcal{O}_Z, \mathcal{O}_Z) = 0$. Hence the Zariski tangent space $T_{[Z]}\mathrm{Hilb}^d(X) \cong \mathrm{Hom}_{X}(\mathcal{I}_Z, \mathcal{O}_Z) \cong \mathrm{Ext}^1_{X}(\mathcal{O}_Z, \mathcal{O}_Z)$ has dimension $2d$. 
\end{proof}
\end{prop}

\begin{rmk}
The preceding proposition does not necessarily prove that the closed point $[Z]$ in $\mathrm{Hilb}^d(X)$ is a smooth point, as it is unclear whether $\mathrm{Hilb}^d(X)$ locally at $[Z]$ has dimension smaller than $2d$.
\end{rmk}
 
The key construction is given in the following proposition, which produces a family of subschemes from a family of reflexive modules. We write $X^0 = X \setminus \{p\}$ and $\mathfrak{X} = X \times \mathbb{A}^1$ for the isotrivial 1-parameter family of $X$ with two projections $\pi_X: \mathfrak{X} \to X$ and $\pi_{\mathbb{A}^1}: \mathfrak{X} \to \mathbb{A}^1$.

\begin{prop}\label{family}
Let ${M_t}$ be a family of reflexive $\mathcal{O}_{X}$-modules of constant rank $r$ over $\mathbb{A}^1$. Suppose $M_0$ is the syzygy module of the ideal $I_{0}$ of a zero-dimensional subscheme $Z_0$ of $X$ with $\mathrm{Supp}(Z_0) = \{p\}$. Then there exists an open neighborhood $U$ of $0 \in \mathbb{A}^1$ such that for any $t \in U$ the corresponding module $M_t$ is the syzygy module of the ideal $I_t$ of a zero-dimensional subscheme $Z_t$ of $X$ with the same length as $Z_0$. 
\begin{proof}
By assumption, there is a short exact sequence of $R$-modules:
\[
0 \to M_0 \xrightarrow{\phi_0} \mathcal{O}_X^{r + 1} \to I_0 \to 0.
\]
We consider the following sequence on $\mathfrak{X}$:
\[
0 \to F_1 \xrightarrow{\phi} F_0 \to \mathcal{Q} \coloneqq \mathrm{Coker}(\phi)   \to 0,
\]
which satisfies the following properties:
\begin{itemize}
\item[1. ] $F_0$ is a free $\mathcal{O}_X[t]$-module of rank $r + 1$;
\item[2. ] $F_1$ is a reflexive $\mathcal{O}_X[t]$-module of rank $r$ with $\mathrm{depth}_q(F_1) = 3$ for any closed point $q \in \mathfrak{X}$;
\item[3. ] the restriction of $\phi$ to the closed fiber $X_0$ over $0 \in \mathbb{A}^1$ is $\phi_0$.
\end{itemize}
The rest of the proof will be a series of restrictions in the affine line to an open neighborhood of 0 such that in each fiber the module $\mathcal{Q}$ restricts to an ideal of a subscheme of $X_t = X$ with constant length as $Z_0$.

To begin with, note that the support of the torsion part of $\mathcal{Q}$ in $\mathbb{A}^1$ is disjoint from 0. We let $U_1 \subset  \mathbb{A}^1$ be the complement of the torsion part of $\mathcal{Q}$ and let $\mathfrak{X}_1 = X \times U_1$ be the restricted family. Now we can assume that $\mathcal{Q}$ is a rank 1 torsion free sheaf on $\mathfrak{X}_1$ such that $\mathcal{Q} \mid_{X_0} \cong I_0$.

The double dual $\mathcal{Q}^{**}$ of $\mathcal{Q}$ with respect to $R[t]$ is a rank 1 reflexive $R[t]$-module. Hence there exists a Weil divisor $D$ of $\mathfrak{X}_1$ with $\mathcal{Q}^{**} \cong \mathcal{O}_{\mathfrak{X}_1}(D)$. Since $\mathfrak{X}$ is a trivial family of $X$ over the affine line, we have isomorphism of their first Chow groups, which surjects onto the Chow group of $\mathfrak{X}_1$:
\[
CH^1(X) \cong CH^1(\mathfrak{X}) \to CH^1(\mathfrak{X}_1) \to 0.
\]
Therefore there exists a Weil divisor $D_0$ of $X^0$ such that $\mathcal{O}_{\mathfrak{X}_1}(D) \cong \pi_X^*\mathcal{O}_{X}(D_0)$. Restricting this identification to $X^0$ we see that $D_0$ is homologous to 0, so it is with $D$. So we have $\mathcal{Q}^{**} \cong \mathcal{O}_{\mathfrak{X}_1}$.

We can write $\mathcal{Q} = \mathcal{I}_W$ as the ideal sheaf of a closed subscheme $W$ of $\mathfrak{X}_1$. This subscheme $W$ is not necessarily equidimensional and hence not finite flat over $U_1$. We will further restrict in $U_1$ to where $W$ become finite flat over the base of degree equal to the length of $Z_0$. First let $W_1$ be the union of the 2-dimensional components of $W$ (possibly empty). We see that $W_1$ is disjoint from $X_0$.

Let $\mathfrak{X}_2 \coloneqq \mathfrak{X}_1 \setminus W_1$. Restricting on $\mathfrak{X}_2$ we can assume that $W$ is a one-dimensional scheme (not necessarily reduced or irreducible). We have a presentation of $\mathcal{I}_W$:
\[
0 \to F_1 \xrightarrow{\phi} F_0 \to \mathcal{I}_W \to 0.
\]
By condition 3 on the depth of $F_1$ above, we conclude that $W$ is a Cohen-Macaulay curve. Moreover, there could be components of $W$ which are contracted by $\pi_{\mathbb{A}^1}$. These vertical components are not contracted to the closed point 0. By further restricting to a possibly smaller open neighborhood $U_2 \subset U_1$ of 0 we can assume that $W$ does not contain any vertical components and hence flat over $U_2$.

Write $\psi: W \to U_2$ for the restriction of the projection $\pi_{\mathbb{A}^1}$ from $\mathfrak{X}_2$. Note that $\psi^{-1}(0)_{\mathrm{red}} = \{(p, 0)\}$, and that both $\mathfrak{X}_2$ and $U_2$ are affine schemes. We can apply Zariski's main theorem to obtain a factorization of $\psi$ through $\tilde{\psi}: Z \to U_2$ and we can find an open neighborhood $\mathcal{U}$ of $(p, 0)$ in $W$ such that the induced morphism $\mathcal{U} \to Z$ is an open immersion. 
\end{proof}
\end{prop}

\begin{theorem}\label{finitehd}
Suppose $(X, p)$ is the germ of a rational double point surface singularity with $X = \mathrm{Spec}(R)$, the closed point $p \in X$ being the only singular point. Suppose $Z_0$ is a closed subscheme of $X$ supported at $p$ with ideal $I_0$. We assume that $I_0$ has a length 2 free resolution on $X$. Then $Z$ is smoothable. 
\begin{proof}
The proof is streamlined by a concrete construction of a flat family smoothing the subscheme $Z$. Write $X_0 = X \setminus \{p\}$ and $\mathfrak{X} = X \times \mathbb{A}^1$ for the isotrivial 1-parameter family of $X$ with two projections $\pi_X: \mathfrak{X} \to X$ and $\pi_{\mathbb{A}^1}: \mathfrak{X} \to \mathbb{A}^1$. Suppose $I_0$ has a minimal $R$-free resolution of the form
\begin{equation}\label{fhd}
0  \to R^{\oplus r - 1} \xrightarrow{\phi_0} R^{\oplus r} \to I_0 \to 0,
\end{equation}
where $\phi_0$ is a $r \times (r - 1)$ matrix with all entries $f_{ij}$ in the maximal ideal $\mathfrak{m}$ of $R$. Now we construct a three short exact sequence of $R[t]$-modules as in the proof of the preceding proposition which restricts to (\ref{fhd}):
\begin{equation}\label{fhdt}
0  \to R[t]^{\oplus r - 1} \xrightarrow{\phi} R[t]^{\oplus r} \to \mathcal{Q} \to 0.
\end{equation}
By the previous construction we can assume that $\mathcal{Q} = \mathcal{I}_W$ for some closed subscheme $W$ of $\mathfrak{X}$ finite and flat over an open neighborhood of $0 \in \mathbb{A}^1$.

Now we need to deform (\ref{fhdt}) so that the restriction to a general fiber is the free resolution of another ideal. We define a 1-parameter family of matrices: $\tilde{\phi} \coloneqq \phi + tI_{r - 1}'$, where $I_{r - 1}' = \begin{bmatrix}
I_{r - 1} \\
0
\end{bmatrix}$.

Note that for $t = 0$ we have $\tilde{\phi}_0 = \phi_0$, and for any $t \neq 0$, the top $(r - 1) \times (r - 1)$-minor of $\tilde{\phi}$ is invertible since its determinant is $t^{r - 1}$. Consequently, the cokernel of $\tilde{\phi}_t$ for $t \neq 0$ is the ideal of a closed subscheme of $X$ of length $d$ which does not have $p$ in its support. 
\end{proof}
\end{theorem}

\begin{rmk}
See \cite[Theorem 2.2, 3.2]{Y04} for results on degenerations of modules in a more general setting, and \cite[Remark 2.5]{Y04} for a statement on specializations of an extension class. 
\end{rmk}

Next we provide with a class of examples of subschemes of finite homological dimension, for which the assumption is clear in the case where the surface is smooth (\cite[Lemma 2]{EL99}).  

\begin{prop}Suppose the syzygy module $\Omega(I_Z)$ of $I_Z$ is given by a minimal presentation
\[
0 \to \Omega(I_Z) \to R^{r + 3} \to I_Z \to 0,
\]
for some integer $r \geq 0$, i.e., $\Omega(I_Z)$ has free rank zero. Suppose $\dim_{\mathbb{C}}\mathrm{Ext}^1_R(I_Z, R) = r + 2$, then $I_Z$ has homological dimension 1 over $R$. 
\begin{proof}
Suppose $\eta_1, \dots, \eta_{r + 2}$ form a set of generators of $\mathrm{Ext}^1_R(I_Z, R)$. Then one can form the extension $\eta \coloneqq (\eta_1, \dots, \eta_{r + 2})$ for some $R$-module $M(I_Z)$:
\[
\eta: 0 \to R^{r + 2} \xrightarrow{\phi} M(I_Z) \to I_Z \to 0.
\]
Applying $\mathrm{Hom}_R(\bullet, R)$ to $\eta$, we have 
\[
0 \to R \to (M(I_Z))^* \to R^{r + 2} \xrightarrow{\gamma} \mathrm{Ext}^1_R(I_Z, R) \to \mathrm{Ext}^1_R(M(I_Z), R) \to 0.
\]
By construction, $\gamma: R^{r + 2} \to \mathrm{Ext}^1_R(I_Z, R)$ is surjective. Hence $\mathrm{Ext}^1_R(M(I_Z), R) = 0$. Then $M(I_Z)$ is maximal Cohen-Macaulay. In fact, Herzog and K\"{u}hl (\cite{HK}) show that $M(I_Z) \cong \Omega(I_Z)^* \oplus R$.

The $r + 2$ components of the map $\phi: R^{r + 2} \to M(I_Z)$ projecting to $\Omega(I_Z)^*$ defines a map $\tilde{\phi}: \Omega(I_Z) \to R ^{r + 2}$ of rank $r + 2$. One can take the push-out diagram 

\begin{center}
\includegraphics{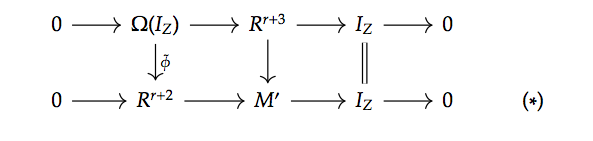}
\end{center} 
 
The bottom row of the diagram is denoted by $(*)$.
Now the two extensions $\eta$ and $(*)$ differ possibly by an automorphism of $R^{r + 2}$ since $\tilde{\phi}$ has full rank. Comparing $\eta$ with $(*)$, we get the following diagram

\begin{center}
\includegraphics{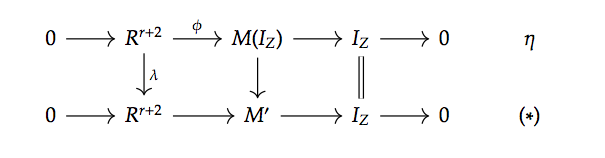}
\end{center} 

By the short-five lemma, the middle vertical map is an isomorphism: $M(I_Z) \cong M'$. Combining this with the previous diagram we have

\begin{center}
\includegraphics{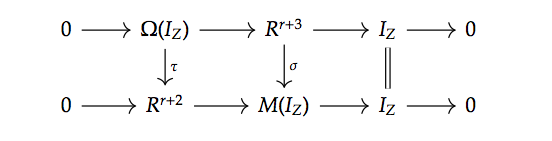}
\end{center} 

Since $\tau$ has full rank, by the Snake Lemma we see that $\ker(\tau) = 0$. Now suppose the cokernal of the left two vertical maps $\tau$ and $\rho$ is $K$. Then $K$ is an $R$-module of finite length. Applying $\mathrm{Hom}_R(\bullet, R)$ to the left two columns we get $\mathrm{Ext}^1(\Omega(I_Z), R) \cong \mathrm{Ext}^2(K, R)$. But since $\Omega(I_Z)$ is maximal Cohen-Macaulay, this ext-group is zero. By Matlis duality, $K = 0$. Hence $\Omega(I_Z)$ is a $R$-free. 
\end{proof}
\end{prop}

The rest of the section is devoted to the study of deformations of reflexive modules which appear as syzygies of zero-dimensional subschemes on the surface. Roughly speaking, if there is a short exact sequence of reflexive modules 
\[
0 \to M \to M'' \to M' \to 0
\]
then the direct sum $M \oplus M'$ can be considered as a specialization of $M''$ by trivializing the extension. On the other hand, for a fixed pair of modules $M$ and $M'$, the $\mathbb{C}$-vector space $\mathrm{Ext}_R^1(M', M)$ is stratified into finitely many strata according to the isomorphism classes of the middle term, and there is a ``generic extension'', i.e., a dense subset of $\mathrm{Ext}_R^1(M', M)$ which all have the same middle term, denoted by $E(M', M)$. Consequently, all of the extensions can be thought as specializations of the generic extension.

To proceed, we first compute the dimensions of the ext-groups of pairs of reflexive modules. Such computations appear in \cite{G80} and in a broader context in \cite{IM}. Suppose $M_1, \dots, M_n$ are the distinct isomorphism classes of indecomposable reflexive $R$-modules. For each $M_i$ there are at least two natural short exact sequences ending with $M_i$: one is the presentation after choosing a minimal set of generators of $M_i$:
\begin{equation}\label{presentation}
0 \to M_i^* \to R^{\oplus 2 \mathrm{rk} (M_i)} \to M_i \to 0
\end{equation}
where $M_i^*$ is the $R$-dual of $M_i$. The other is the Auslander-Reiten sequence, which has both endings $M_i$:
\begin{equation}\label{arsequence}
0 \to M_i \to E(M_i) \to M_i \to 0,
\end{equation}
where $E(M_i)$ is the direct sum of all the indecomposable modules corresponding to the vertices adjacent to the one for $M_i$ in the extended Dynkin diagram. For the reader's convenience, we briefly reproduce the construction of both cases below. For a rigorous exposition on the sequence (\ref{arsequence}) in an appropriate generality, see \cite{W88}. For the coincidence between the McKay quiver associated the the finite group $G \leq \mathrm{SL}_2(\mathbb{C})$ and the Auslander-Reiten quiver on the set of indecomposable reflexive modules we refer the reader to Auslander's original paper (\cite{A86}). 

An extended Dynkin diagram is obtained from the usual Dynkin diagram by adding one vertex corresponding to the trivial module $R$ at the appropriate place. In the diagrams below, we label the vertices in the $A_n$ case $1, \dots, n$ from the left to the right, and in the $D_n$ case $1, \dots, n$ from the left to the right (the last two vertices $n - 1$ and $n$ are not distinguished). For $A_n$, one puts the extra vertex on top of the chain of $n$ existing vertices and close the loop by adding edges from vertices 1 and $n$ to this $(n + 1)$-st vertex. For $D_n$, the $(n + 1)$-st vertex is added to vertex 2 on the left tail so that the picture would become symmetric. For $E_6$ the extra vertex is placed and connected by an edge on top of vertex 4; for $E_7$ and $E_8$ the extra vertex is connected to vertex 1 (labelled by a circle).  
\begin{center}
\includegraphics{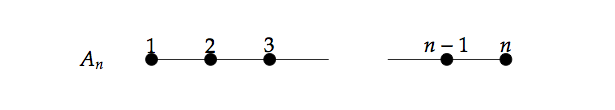}

\includegraphics{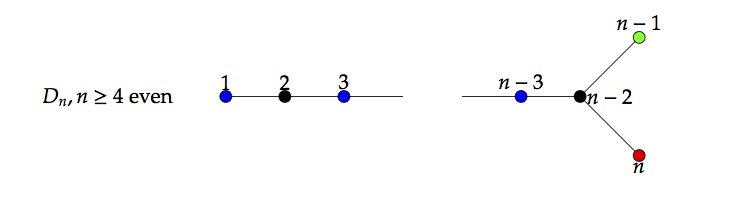}

\includegraphics{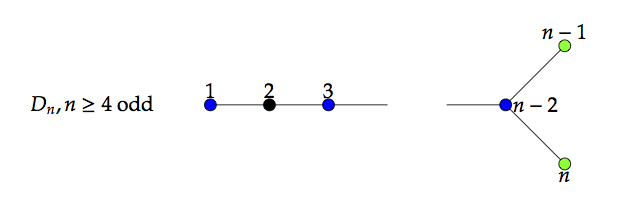}

\includegraphics{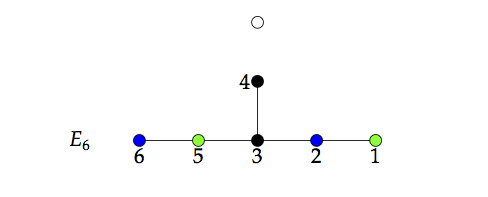}

\includegraphics{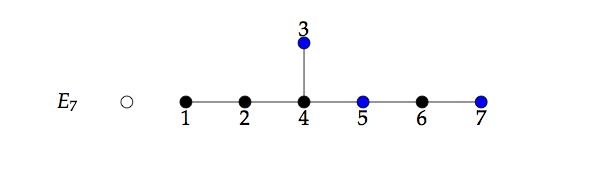}

\includegraphics{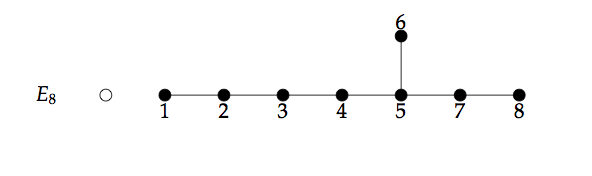}

\small{Dynkin diagrams of $ADE$ singularities \\(in each cases of $D_n$ and $E_n$ nodes of the same color correspond to modules \\with the same determinant in the local class group of the singularity)}
\end{center}

To explain sequence (\ref{presentation}), we have the following
\begin{lemma}The $R$-dual of each finite rank reflexive module is listed as below:
\begin{itemize}
\item[$A_n$.] $M_i \cong M^*_{n + 1 - i}$ for $i = 1, \dots n$. 
\item[$D_n$.] $M_i \cong M_i^*$ for $i = 1, \dots, n - 2$. If $n$ is even, then $M_{n - 1}$ and $M_n$ are also self-dual; if $n$ is odd, then they are dual to each other. 
\item[$E_n$.] In $E_6$, $M_1 \cong M_6^*, M_2 \cong M_5^*$, and $M_4$ and $M_3$ are self-dual. In $E_7$ and $E_8$ every indecomposable module is self-dual.
\end{itemize}
\end{lemma}

Let $\widetilde{\Omega_{X}}$ be the reflexive hull of the sheaf of holomorphic 1-forms on the open surface $X_0$ extended to $X$. Then $\widetilde{\Omega_{X}}$ is a reflexive $\mathcal{O}_X$-module. There is the \textit{fundamental sequence} on $X$:
\begin{equation}\label{fundamental}
0 \to \omega_{X} \to \widetilde{\Omega_{X}} \to \mathcal{O}_{X} \to \mathbb{C} \to 0.
\end{equation}
Suppose $M$ is a non-trivial indecomposable reflexive module on $X$, then the Auslander-Reiten sequence ending with $M$ can be obtained by tensoring the sequence (\ref{fundamental}) with $M$ and taking reflexive hulls, resulting with a sequence of the form
\[
0 \to \tau(M) \to (\widetilde{\Omega_{X}}\otimes M)^{**} \to M \to 0,
\]
where $\tau(M) = (\omega_{X} \otimes M)^{**}$ is called the \textit{Auslander-Reiten transpose} of $M$. In the current case, since the surface is Gorenstein $\tau(M) \cong M$, and the middle term was denoted by $E(M)$ in (\ref{arsequence}). 

However, there are in fact many more extensions of reflexive modules. A combinatorial way to find the dimensions of $\mathrm{Ext}^1_R(M_i, M_j)$ for any pair of modules $M_i$ and $M_j$ is developped by Iyama and Wemyss (\cite[Theorem 4.4, 4.5]{IM}). In this article we compute some examples following (\cite{IM}). 

\begin{eg}($A_n$)
In the $A_n$ case, for any integers $1 \leq a \leq b \leq n$ we have $\dim_{\mathbb{C}}\mathrm{Ext}_R^1(M_b, M_a) = \dim_{\mathbb{C}}\mathrm{Ext}_R^1(M_a, M_b) = \max\{a, n + 1 - b\}$.
\end{eg}

\begin{eg}($D_6$)
Say we want to compute $\dim_{\mathbb{C}} \mathrm{Ext}^1_R(M_2, M_1)$. We will produce a chain of tuples of the indices on the Dynkin diagram. We start with the tuple $(2)$ indicating the extensions should end at $M_2$. The next tuple is the collection of all indices adjacent to $2$, namely $(1, 3)$, and we write $(2) \mapsto (1, 3)$. The next tuple is the collection of all adjacent vertices of $(1, 3)$ counting multiplicities subtracting the previous one, hence it is $(2, 2, 4) - (2) = (2, 4)$ and we write $(2) \mapsto (1, 3) \mapsto (2, 4)$. This chain keeps going:
\[
(2) \mapsto (1, 3) \mapsto (2, 4) \mapsto (3, 5, 6) \mapsto (4, 4)  \mapsto (3, 5, 6) \mapsto (2, 4) \mapsto (1, 3) \mapsto (2) \mapsto 0.
\]
Eventually the sequence will stop at $0$ and we count how many 1's showing up in this entire sequence. There are two 1's, so $\dim \mathrm{Ext}^1_R(M_2, M_1) = 2$. This method works for any $ADE$ singularities.
\end{eg}

\begin{eg}($D_4$)
The same calculation as for $D_6$ says that $\dim_{\mathbb{C}} \mathrm{Ext}^1_R(M_i, M_i) = 2$ for any $i = 1, 3, 4$. 
\end{eg}

Knowing the dimensions of the extensions for each pair of reflexive modules gives a preliminary estimation for the deformations of the modules, and we need to find the stratification within each Ext-group and to determine the most general extensions. For a fixed singularity and a fix pair of reflexive modules $M_i$ and $M_j$ on it, we declare a partial order on $\mathrm{Ext}^1_R(M_i, M_j)$: two extensions satisfy $\eta > \tau$ if $\tau$ is a specialization of $\eta$.

\begin{lemma}
In each complex vector space $\mathrm{Ext}^1_R(M_i, M_j)$ there is a unique maximal element, corresponding to the ``generic extension''.
\begin{proof}
This should be a general homological statement. Suppose to the contrary that there are two non-isomorphic extensions $\eta, \tau \in \mathrm{Ext}^1_R(M_i, M_j)$ such that each has an open neighborhood $U_1$ and $U_2$ in $\mathrm{Ext}^1_R(M_i, M_j)$ in its usual Euclidean topology. Then we connect $\eta$ and $\tau$ by a line segment and construct a homotopy from one to the other. This line segment will intersect with $U_1$ and $U_2$ both on an open subset. This is impossible since a one-dimensional subspace of $\mathrm{Ext}^1_R(M_i, M_j)$ can have only one generic class. 
\end{proof}
\end{lemma}

In the $A_n$ case, regarding the $M_i$'s as fractional ideals over $\mathcal{O}_{X}$, namely, $M_i = \mathcal{O}_{X}\langle 1, u^i/v^{n + 1 - i}\rangle$, we can explicitly express the mutual map between $M_i$ and $M_{i + 1}$: one direction is given by inclusion $M_{i + 1} \hookrightarrow M_i$, and the other direction $M_{i} \to M_{i + 1}$ sends 1 to 1 and multiplies $u^i/v^{n + 1 - i}$ by $uv$. By composing morphisms between the modules, we have:

\begin{lemma}[$A_n$]\label{maps}
Let $X_n$ be the normal surface with an isolated $A_n$ singularity. For any integers $1 \leq a \leq b \leq n$ there is the following exact sequence:
\[
0 \to M_a \to M_{a - i} \oplus M_{b + i} \to M_b \to 0,
\]
where $i \in \{0, \dots, \max\{a, n + 1 - b\} \}$. Any extensions of $M_b$ by $M_a$ has such form. Moreover, the following sequence is generic:
\[
0 \to M_a \to \mathcal{O}_{X_n} \oplus M_{\overline{a + b}} \to M_b \to 0,
\]
where $\overline{a + b} \equiv a + b (\mod n + 1)$, $\overline{a + b} \leq n + 1$. 
\end{lemma}

\begin{lemma}[$D_4$]\label{d4}
We have 
\[
\dim_{\mathbb{C}} \mathrm{Ext}^1_R(M_i, M_i) = 2
\]
for $i = 1, 3, 4$. Equivalently, the Auslander-Reiten sequence and the minimal presentation sequence span the space $\mathrm{Ext}^1_R(M_i, M_i)$ for $i = 1, 3, 4$. Moreover, there is an open dense subset of $\mathrm{Ext}^1_R(M_i, M_i)$ giving rise to extensions whose middle term is $R^{\oplus 2}$.
\end{lemma}
\begin{proof}
We take $\eta_1, \eta_2 \in \mathrm{Ext}^1_R(M_i, M_i)$ and their respective representatives:
\begin{align}
& \eta_1:  0  \to M_i \to M_2 \to M_i \to 0 \\
& \eta_2:  0  \to M_i \to R^{\oplus 2} \to M_i \to 0.
\end{align}
The lemma will be proved if we show that the sum of the representatives of $\eta_1$ and $\eta_2$ will represent a class whose middle term is $R^{\oplus 2}$. Without lose of generality, we can take $i = 1$. First there is the following pull-back diagram:

\begin{center}
\includegraphics{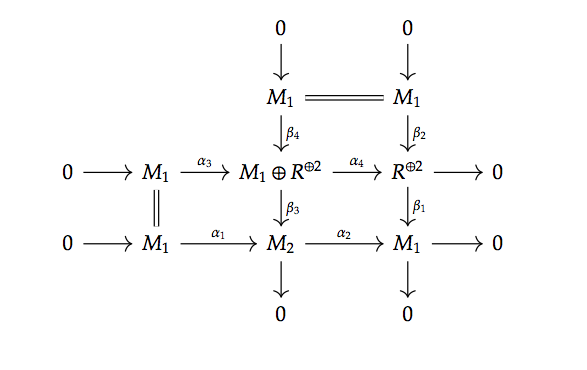}
\end{center}

where the bottom row and the rightmost column represent $\eta_1$ and $\eta_2$. The term in the center is $M_1 \oplus R^{\oplus 2}$ since the middle row must be a trivial extension of $R^{\oplus 2}$ by $M_1$. In particular, $\alpha_3(m) = (m, 0)$ and $\beta_4(m) = (0, \beta_2(m))$ for any $m \in M_1$. Note that $M_1 \oplus R^{\oplus 2}$ contains three copies of $M_1$: via $\alpha_1$ and $\beta_2$, and also there is the skew-diagonal coming from the pullback via $\lambda \coloneqq (\mathrm{Id}, -\beta_2)$. Then by definition, the sum $\eta_1 + \eta_2$ is the extension whose middle term is the quotient of $M_1 \oplus R^2$ by the skew-diagonal, denoted by $N$. Hence the following is exact:
 \[
 0 \to M_1 \xrightarrow{\lambda} M_1 \oplus R^{\oplus 2}  \to N \to 0.
 \] 
 
To determine $N$, we choose a minimal set of generators of $M_1$ over $R$ so that $M_1 \oplus R^{\oplus 2}$ admits a surjection from $R^{\oplus 4}$. Then one can take the pullback to get 

\begin{center}
\includegraphics{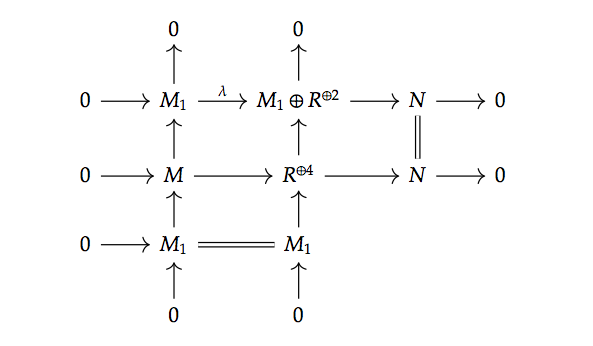}
\end{center}
 
Note that the middle term in the left column $M$ is maximal Cohen Macaulay, so it can be either $M_1 \oplus M_1, M_2$ or $R^{\oplus 2}$. Also note that the middle column is not minimal, meaning that the trivial part $R^{\oplus 2} \to R^{\oplus 2}$ can be taken out, and it will coincide with the left column by short-five lemma. Hence, the middle term in the left column is isomorphic to $R^{\oplus 2}$. Then $N \cong R^{\oplus 2}$ as well. Consequently, the sum $\eta_1 + \eta_2 \in \mathrm{Ext}^1_R(M_1, M_1)$ can be represented by a sequence of the form
 \[
\eta_1 + \eta_2:  0  \to M_1\to R^{\oplus 2} \to M_1 \to 0.
 \] 
This summation respects the $\mathbb{C}$-vector space structure of $\mathrm{Ext}^1_R(M_1, M_1)$. Hence in the stratification of $\mathrm{Ext}^1_R(M_1, M_1)$ there is an open dense part represented by extensions whose middle term is $R^{\oplus 2}$. In particular, one can deform $M_2$ to $R^{\oplus 2}$. 
\end{proof}

The following general statement holds true, and the proof is the same as the preceding lemma.
\begin{prop}
If there exists a class $\eta \in \mathrm{Ext}^1_R(M_i, M_j)$ represented by 
\[
0 \to M_j \to R^{\oplus (\mathrm{rk} M_i + \mathrm{rk} M_j)} \to M_i \to 0.
\]
Then $\eta$ is the generic extension.
\end{prop}

Following Iyama-Wemyss (\cite[Theorem 4.8, 4.9]{IM}) we can find a basis for $\mathrm{Ext}^1_R(M_i, M_j)$. The conceptual construction of \textit{ladders} in $\tau$-categories was introduced by Iyama (\cite{I05}), which was made more concrete in \cite{IM} for the study of special Cohen-Macaulay modules over general two-dimensional quotient singularities. Here we apply their results in the $ADE$ case. 

\begin{eg}
We take the rank 2 module $M_8$ on the $E_8$ singularity as an example. Besides the syzygy sequence and the Auslander-Reiten sequence starting at $M_8$, the following sequences also start at $M_8$:
\begin{align*}
0 \to M_8 \to & M_5 \to M_7 \to 0;\\
0 \to M_8 \to M_4 & \oplus M_6 \to M_5 \to 0.
\end{align*}
There are other short exact sequences or reflexive modules starting at $M_8$, however, whose middle terms are less explicit. 
\end{eg}

\section{The Hilbert schemes of points on surfaces with rational double points}
The following lemma provides with a necessary condition for a reflexive module to be the syzygy of a zero-dimensional subscheme supported at the singularity $(p \in X, R = \mathcal{O}_{X, p})$. We call these modules \textit{syzygy modules}.

\begin{lemma}\label{firstchern}
Suppose $Z$ is a zero-dimensional subscheme of $X$ supported at $p$ with ideal $I$ of $R$ and first syzygy $\Omega$. Then the determinant of $\Omega$ is trivial in the local class group of $R$.
\begin{itemize}
\item[$A_n$.] Suppose $\Omega = R^{\oplus a} \oplus M_1^{\oplus a_1} \oplus \dots \oplus M_n^{\oplus a_{n}}$. Then $m(I) \coloneqq a_1 + 2a_2 + \dots + na_n$ is divisible by $n + 1$.
\item[$D_{2k}$.] Suppose $n = 2k$ is even. Then $\Omega$ is possibly the direct sum of multiples of: $R, M_{2i - 1}^{\oplus 2}$ for $i = 1, \dots, k$,  $M_{2j}$ for $j = 1, \dots, k - 1$, $M_{2k}^{\oplus 2}, M_{2i - 1} \oplus M_{2j - 1}$ for $1 \leq i \neq j \leq k - 1$, and $M_{2i - 1} \oplus M_{2k - 1} \oplus M_{2k}$ for $i = 1, \dots, k - 1$.
\item[$D_{2k + 1}$.] Suppose $n = 2k + 1$ is odd. Then $\Omega$ is possibly the direct sum of multiples of: $R, M_{2i - 1}^{\oplus 2}$ for $i = 1, \dots, k$, $M_{2k}^{\oplus 4}, M_{2k + 1}^{\oplus 4}, M_{2k} \oplus M_{2k + 1}, M_{2i - 1} \oplus M_{2j - 1}$ for $1 \leq i \neq j \leq k - 1$, and $M_{2j}$ for $j = 1, \dots, k - 1$.
\item[$E_6$.] $\Omega$ is possibly the direct sum of multiples of: $M_1 \oplus M_6, M_2 \oplus M_5, M_1 \oplus M_2, M_3, M_4$.
\item[$E_7$.] $\Omega$ is possibly the direct sum of multiples of: $M_1, M_2, M_4, M_6, M_3^{\oplus 2}, M_5^{\oplus 2}, M_7^{\oplus 2}, M_3 \oplus M_5, M_7 \oplus M_5, M_3 \oplus M_7$.
\item[$E_8$.] Any reflexive module can be a syzygy of a zero-dimensional subscheme from the consideration of the triviality of the determinant.
\end{itemize}
\begin{proof}
All the cases follow immediately by looking at the local class group of each singularity.
\end{proof}
\end{lemma}

\begin{prop}\label{key}
Any syzygy module can be generalized to a free module of the same rank, namely, there is a connected family of reflexive modules for which the general member of the family is free, and the special member is the syzygy module. 
\begin{proof}
We proceed in two cases depending on whether $M$ is indecomposable for each singularity. 

($A_n$). Suppose the syzygy module is $M =  R^{\oplus a} \oplus M_1^{\oplus a_1} \oplus \dots \oplus M_n^{\oplus a_{n}}$. We can (but not necessarily) choose to deform $M$ in the following order. First, for any $M_i$ appearing in $M$ with $a_i \geq 2$, we deform $M_i^{\oplus 2}$ to $R \oplus M_{\overline{2i}}$. Each such step increases the free rank of $M$ by 1 and generalize the input module. After finitely many steps, there is no non-free summand with multiplicity greater than 1 in the generalization (still denoted by $M$). Second, we start with $M_i$ and $M_j$ such that $i$ and $j$ are the smallest two indices for summands of $M$. Using the sequence $0 \to M_i \to R \oplus M_{\overline{i + j}} \to M_j \to 0$ we see that $M_i \oplus M_j$ can be generalized to $R \oplus M_{\overline{i + j}}$. Such a step also increases the free rank of $M$ by 1. We can repeat these two steps till the eventual outcome is free. 

($D_4$). It suffices to consider $M_i^{\oplus 2}$ for $i = 1, 3, 4$ and $M_2$. First note that $M_i$ is self-dual for $i = 1, 3, 4$. Hence the presentation sequence $0 \to M_i \to R^{\oplus 2} \to M_i \to 0$ realizes the generalization of $M_i^{\oplus 2}$ to $R^{\oplus 2}$. The case of $M_2$ follows from Lemma \ref{d4}.

($D_6$). We work out $D_6$ explicitly. The modules $M_1, M_2, M_4, M_5, M_6$ are self-dual. Hence $M_1^{\oplus 2}, M_5^{\oplus 2}, M_6^{\oplus 2}$ can be generalized to $R^{\oplus 2}$. Next we compute that $\dim_{\mathbb{C}} \mathrm{Ext}^1_R(M_1, M_1) = 2$, and hence $M_2$ can be generalize to $R^{\oplus 2}$. Two other easy cases are $M_4$ and $M_1 \oplus M_5 \oplus M_6$. We compute that $M_4$ can be generalized to $R^{\oplus 2}$ using extensions in $\mathrm{Ext}^1_R(M_5, M_5)$, which is 2-dimensional and the generic one has $R^{\oplus 2}$ in the middle and the special one has $M_4$ in the middle. For $M_1 \oplus M_5 \oplus M_6$, we note the sequence $0 \to M_1 \to R \oplus M_5 \to M_6 \to 0$, which allows $M_1 \oplus M_5 \oplus M_6$ to be generalized to $R \oplus M_5^{\oplus 2}$, and further to $R^{\oplus 3}$. For $M_3 \oplus M_5 \oplus M_6$, this is the middle term in the Auslander-Reiten sequence for $M_4$: $0 \to M_4 \to M_3 \oplus M_5 \oplus M_6 \to M_4 \to 0$. The same argument shows that $M_3 \oplus M_5 \oplus M_6$ generalizes to $R^{\oplus 4}$ by using $0 \to M_4 \to R^{\oplus 4} \to M_4 \to 0$.

For the remaining case $M_1 \oplus M_3$, we start with two classes in $\mathrm{Ext}^1_R(M_1, M_1)$ represented by 
\[ 
0 \to  M_1 \to R^{\oplus 2} \to M_1 \to 0, \quad \textrm{ and } \quad 0 \to  M_1 \to M_2 \to M_1 \to 0.
\]
Taking the pull-back of them we get the diagram below:

\begin{center}
\includegraphics{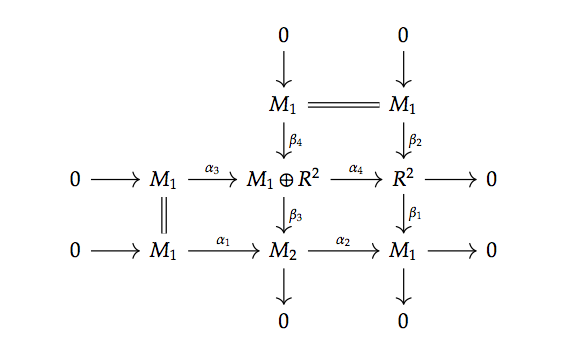}
\end{center}
 
In particular, the middle column is non-trivial: $0 \to M_1 \to M_1 \oplus R^{\oplus 2} \to M_2 \to 0$. Next we take the Auslander-Reiten sequence of $M_2$: $0 \to M_2 \to M_1 \oplus M_3 \oplus R \to M_2 \to 0$.Taking the pull-back of these two sequences we have a diagram below:

\begin{center}
\includegraphics{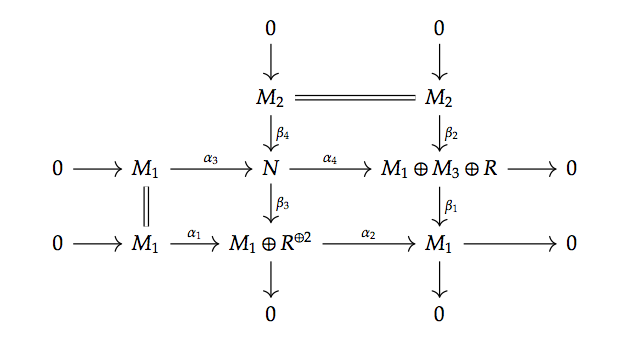}
\end{center}
 
Note that the module $N$ in the center of the diagram is maximal Cohen-Macaulay since $\mathrm{Ext}^i(N, R) = 0$ for all $i \geq 1$ by looking at the long exact sequence after applying $\mathrm{Hom}_R(\bullet, R)$ to either the middle row or the middle column. Now the idea is to look for all possibilities for $N$ to fit in the middle column. We compute that $\dim \mathrm{Ext}^1_R(M_1, M_2) = 2$ and hence $N$ is either $M_1 \oplus R^{\oplus 4}, M_1 \oplus M_2 \oplus R^{\oplus 2}, M_1^{\oplus 2} \oplus M_3 \oplus R$. Among the three choices, $M_1^{\oplus 2} \oplus M_3 \oplus R$ will make the middle row trivial, and $M_1 \oplus M_2 \oplus R^{\oplus 2}$ will trivialize the middle column. So we see that the most general one is $M_1 \oplus R^{\oplus 4}$ and the middle row becomes $0 \to M_1 \to M_1 \oplus R^{\oplus 4} \to M_1 \oplus M_3 \oplus R \to 0$. The inclusion of $M_3$ into the right term $M_1 \oplus M_3 \oplus R$ induces either 
\[
0 \to M_1 \to R^{\oplus 3} \to M_3 \to 0 \quad \textrm{ or } \quad 0 \to M_1 \to M_1^{\oplus 2} \oplus R \to M_3  \to 0.
\]
Either case will generalize $M_1 \oplus M_3$ to something we already know. 

($D_n$ for any even $n$). The only new phenomenon in $D_8$ than $D_6$ is the module $M_4$ whose determinant is trivial. But we can compute that $\dim \mathrm{Ext}^1_R(M_7, M_7) = 4$ and $M_4$ is a possible middle term which can be generalized to $R^{\oplus 2}$. The argument is the same as before. The same proof works for $D_n$ where $n$ is even. 

($D_n$ for any odd $n$). For $D_{2k + 1}$ the new phenomenon comes with the rank 2 module in the middle of the left tail with trivial determinant. For example in $D_7$ the module $M_4$ has trivial determinant. We compute that $\dim \mathrm{Ext}^1_R(M_6, M_7) = 3$, and $M_4$ can be the middle term that generalizes to $R^{\oplus 2}$. 

($E_6$). Note that there are extensions 
\begin{align*}
& 0 \to M_1 \to M_4 \to M_6 \to 0 \\
& 0 \to M_4 \to M_3 \oplus R \to M_4 \to 0 \\
\end{align*}
Hence the same argument as before will take care of all the listed cases except $M_1 \oplus M_2$. We consider the following diagram:

\begin{center}
\includegraphics{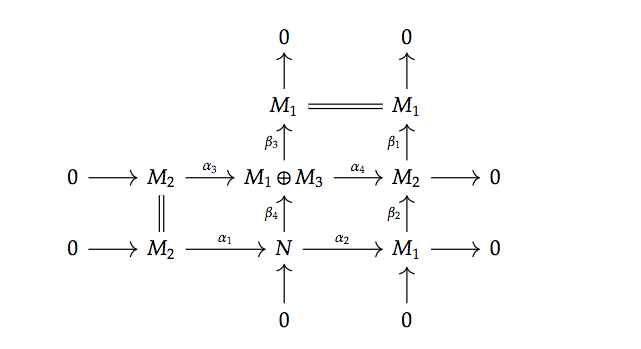}
\end{center}

where we start with the middle row and the right column as the Auslander-Reiten sequence for $M_2$ and $M_1$ respectively. The pull-back of the maps $\alpha_4$ and $\beta_2$ gives the rest of the diagram. The kernel $N$ of the middle column can be considered as the generalization of $M_1 \oplus M_2$, and $M_3$ can be a choice.

($E_7$). First we note that $M_7$ is the only one having rank 1, hence it is self-dual. We also compute that $\dim \mathrm{Ext}^1(M_7, M_7) = 3$, where any extension should have the middle term of rank 2 with trivial determinant. There are only three different such modules other than $M_7 \oplus M_7$, namely $M_6, M_1$ and $R^{\oplus 2}$. As before, we claim that these are all the possible extensions, and the most general one has the form $0 \to M_7 \to R^{\oplus 2} \to M_7 \to 0$. These take care of $M_1, M_7 \oplus M_7$ and $M_6$.

Next we note that $M_1$ is self-dual. The Auslander-Reiten sequence for $M_1$ reads $0 \to M_1 \to R \oplus M_2 \to M_1 \to 0$, whereas the presentation of $M_1$ gives $0 \to M_1 \to R^{\oplus 4} \to M_1 \to 0$. From these two sequences above we see that $M_2$ can be generalized to $R^{\oplus 3}$ by first adding a copy of $R$. Furthermore, $\dim \mathrm{Ext}^1(M_1, M_1) = 4$. So there are extensions $0 \to M_1 \to M_4 \to M_1 \to 0$, and $0 \to M_1 \to M_5 \oplus M_7 \to M_1 \to 0$.

We also see that $M_3$ and $M_5$ are both self-dual, hence $M_3^{\oplus 2}$ and $M_5^{\oplus 2}$ can be generalized to $R^{\oplus 4}$ and $R^{\oplus 6}$ respectively. The two remaining cases are $M_3 \oplus M_7$ and $M_3 \oplus M_5$. The same computation shows that $\dim \mathrm{Ext}^1(M_7, M_3) = 3$, and $\dim \mathrm{Ext}^1(M_5, M_3) = 9$. Hence both cases can be generalized to something we already know.

($E_8$). First we compute that $\dim \mathrm{Ext}^1(M_8, M_8) = 8$. All possible middle terms in extensions of $M_8$ by itself are $R^{\oplus 4}, M_7, M_3, M_6 \oplus R, M_2 \oplus R, M_1^{\oplus 2}, M_1 \oplus R^{\oplus 2}, M_1 \oplus M_8$, and $M_8^{\oplus 2}$ where the last one can only represent the trivial extension. For dimension reason, we see that all of them do show up in the middle of some extension of $M_8$ by itself. In particular, there is a sequence $0 \to M_8 \to M_1 \oplus R^{\oplus 2} \to M_8 \to 0$. Hence to deform $M_1$ we first add a free summand of rank 2 and generalize $M_1 \oplus R^{\oplus 2}$ to $R^{\oplus 4}$. The same argument works for generalizations of $M_2, M_3, M_6$ and $M_7$ to $R^{\oplus 4}$.

For $M_5$ we use the presentation sequence and the Auslander-Reiten sequence of $M_6$ to generalize $M_5$ to $R^{\oplus 6}$. For $M_4$ we look at $\mathrm{Ext}^1(M_8, M_2)$, where there is an extension $0 \to M_2 \to M_4 \to M_8 \to 0$. A surjection $M_4 \to M_8$ can be seen from their matrix factorizations: Suppose the polynomial defining the $E_8$ singularity is $z^2 + x^3 + y^5 \in \mathbb{C}[x, y, z]$. We take matrices 
\[
\phi_8 = \begin{bmatrix}
     & & -x^2 & -y^3 \\
     & & -y^2 & x \\
    x & y^3 & &\\
    y^2 & -x^2 &  &
\end{bmatrix}
\quad 
\mathrm{       and }
\quad 
\phi_4 = \begin{bmatrix}
     & & & & & -x^2 & -y^3 & 0 & 0 & 0 \\
     & & & & & - y^2 & x & 0 & 0 & 0 \\
     & & & & & -y^3 & 0 & -x^2 & xy^2 & -y^4 \\
     & & & & & xy & 0 & -y^3 & -x^2 & xy^2\\
     & & & & & 0 & y^2 & xy & -y^3 & -x^2 \\
     x & y^3 & 0 & 0 & 0 & & & & & \\
     y^2 & -x^2 & 0 & 0 & 0 & & & & & \\
     0 & 0 & x & y^2 & 0 & & & & &\\
     y & 0 & 0 & x & y^2 & & & & & \\
     0 & - y^2 & y & 0 & x & & & & & 
\end{bmatrix}
\]
Then $M_8$ and $M_4$ admit free $R$-resolutions given by matrix factorizations $M(zI_4 - \phi_8, zI_4 + \phi_8)$ and $M(zI_{10} - \phi_4, zI_{10} + \phi_4)$ respectively in the notations of Example \ref{a1}. In particular we see that $\phi_8$ shows as a 4-by-4 minor of $\phi_4$. This allows one to deform $M_4$ to some other module corresponding to a more general extension in $\mathrm{Ext}^1(M_8, M_2)$. 

For $M_8$, we cannot deform it using the preceding argument as there is no rank 1 non-trivial reflexive modules over $E_8$. Instead, we deform it to $M_1$ by fitting both $M_8$ and $M_1$ as kernels of some surjective maps of reflexive modules. By \cite{IM} we find the following short exact sequences:
\begin{align*}
& 0 \to M_1 \to R^{\oplus 2} \oplus M_3 \oplus M_5 \oplus M_8 \xrightarrow{\alpha} M_2 \oplus M_4 \oplus M_7 \to 0; \\
& 0 \to M_8 \to M_1 \oplus M_3 \oplus M_5 \oplus M_8 \xrightarrow{\beta} M_2 \oplus M_4 \oplus M_7 \to 0.
\end{align*}
In the above two sequences, the maps $\alpha$ and $\beta$ restricted to their common summand $M_3 \oplus M_5 \oplus M_8$ are identical, since both are induced by the ladders over Auslander-Reiten sequences. Now we can apply a specialization of $R^{\oplus 2}$ to $M_1$ to specialize $\alpha$ to $\beta$, which in turn induces a specialization of $M_1$ to $M_8$. Hence $M_8$ is trivializable.
\end{proof}

\end{prop}

\begin{theorem}\label{an}
Suppose $X$ is a quasi-projective normal surface with at worst $ADE$ singularities. Then for any positive integer $d$ the Hilbert scheme of $d$ points $\mathrm{Hilb}^d(X)$ is irreducible.
\begin{proof}
The theorem will be proved after several steps of reductions. First, since deformations of reflexive modules at a singularity on $X$ can always be proceeded in some open affine neighborhood of the singularity and each singularity can be treated separately, we are reduced to the case where $X$ is affine with a single $ADE$ singular point $p$: $(X = \mathrm{Spec}(\mathbb{C}[x, y, z]/\langle f \rangle), p)$. Second, it suffices to show that any length $d$ subscheme $Z$ of $X$ supported at $p$ is smoothable. Third, by Theorem \ref{finitehd}, we note that such a punctured subscheme is smoothable if it has finite projective dimension on $X$. Hence we are reduced to prove that any length $d$ subscheme $Z$ on $X$ supported at $p$ is a specialization of a family of length $d$ subschemes of finite projective dimension.

Now suppose that $Z$ is an arbitrary length $d$ subscheme of $X$ supported at $p$ and that the first syzygy module $M_Z \coloneqq \mathrm{Syz}_1(Z)$ of $I_Z$ is not free. In any event, $M_Z$ is a reflexive $\mathcal{O}_X$-module, hence by Lemma \ref{firstchern}, we obtain a direct sum decomposition of $M_Z = \mathcal{O}_{X}^{\oplus a} \oplus M_1^{\oplus a_1} \oplus \dots \oplus M_n^{\oplus a_{n}}$. The proof is proceeded by induction on the rank of the non-free summand in $M_Z$. The strategy is to use Proposition \ref{key} to deform $M_Z$ into the direct sum of the trivial module. In the resulting family of reflexive modules, the free rank of a general member is at least 1 higher than that of the special member $M_Z$. Then one can proceed with the induction. The induction will terminate after finitely many steps of generalization to a free module because as a syzygy module, the first Chern class of $M_Z$ in the local class group of $(X, p)$ must be zero.
 
Now Proposition \ref{family} indicates that the ideal $I_Z$ is the specialization of ideals of finite projective dimension. Equivalently, the closed point $[Z]$ representing the length $d$ subscheme $Z$ is contained in the Zariski closure of the locus in $\mathrm{Hilb}^d(X)$ parameterizing length $d$ subschemes of finite projective dimension. 
\end{proof}
\end{theorem}

\begin{rmk}
There is a natural stratification of $\mathrm{Hilb}^d(X)$ by locally closed subsets, such that within each stratum the first syzygy module of the subschemes are stably isomorphic, i.e., isomorphic up to free summands. For each fixed $d$, the ranks of the syzygy modules are bounded, hence the stratification is finite. Besides the open dense stratum of schemes of finite projective dimension, all other strata are contained in the punctual Hilbert schemes at the singularities. The preceding Proposition \ref{key} and Theorem \ref{an} suggest that there is a partial order on the set of strata given by inclusion in the closure. For each singular point, this is analogous to the closure relation of nilpotent orbits in the nilpotent cone of the Lie algebra whose Weyl group has the same Dynkin diagram as the singularity. But it is unclear to the author what a precise relation is.
\end{rmk}

\begin{rmk}
In the literature the study of deformations of reflexive modules and of matrix factorizations allows the singularity to deform as well, which does not meet our purpose. Nevertheless, the existence of the versal deformation space of reflexive modules over quotient surface singularities can be viewed as a necessary condition for our deformations (\cite{I00}). 
\end{rmk}

\section{Geometric consequences of irreducibility}
The first corollary of the irreducibility of the Hilbert scheme is a smoothness criterion which states that the condition in Proposition \ref{smoothness} is also sufficient:
\begin{theorem}\label{converse}
Suppose $Z$ be a length $d$ subscheme of the surface $X$ with only $ADE$ singularities. Then $\mathrm{Hilb}^d(X)$ is smooth at $[Z]$ if and only if $I_Z$ has finite homological dimension over $X$.
\begin{proof}
One direction is Proposition \ref{smoothness}. To prove the other direction, it suffices to show that if $Z$ does not have finite homological dimension over $X$ then the dimension of the Zariski tangent space $T_{[Z]}\mathrm{Hilb}^d(X)$ is larger than $2d$. Again, we are reduced to the case $X = \mathrm{Spec}(R)$, the affine surface with only one singular point $p$ and $Z$ is supported at $p$. Let $M_1, \dots, M_n$ be the collection of all indecomposable reflexive $R$-modules up to isomorphism.

Suppose the presentation of the ideal $I_Z$ has the form
\begin{equation}\label{presenideal}
0 \to R^{\oplus a} \oplus \left(\bigoplus_{i = 1}^nM_i^{\oplus a_i}\right) \to R^{\oplus r + 1} \to I_Z \to 0,
\end{equation}
with $r = a + \sum_{i = 1}^na_i\mathrm{rk}(M_i)$. Applying $\mathrm{Hom}_{R}(\bullet, R)$ to (\ref{presenideal}) we obtain the following exact sequence. 
\begin{equation}\label{4term}
0 \to \mathrm{Hom}_{R}(I_Z, R) \to \mathcal{O}_{X}^{\oplus r + 1} \to R^{\oplus a} \oplus \bigoplus_{i = 1}^n(M_i^*)^{\oplus a_i} \to \mathrm{Ext}^1_{R}(I_Z, R) \cong \omega_Z \to 0.
\end{equation}

Next applying $\mathrm{Hom}_{R}(\bullet, \mathcal{O}_Z)$ to (\ref{presenideal}), or equivalently, $\mathrm{Hom}_{R}(\bullet, R) \otimes_R \mathcal{O}_Z$, the following sequences of $\mathcal{O}_{Z}$-modules are exact:
\begin{equation}
0 \to \mathrm{Hom}(I_Z, \mathcal{O}_Z) \to \mathcal{O}_Z^{\oplus r + 1} \to \mathcal{O}_Z^{\oplus a} \oplus \left( \bigoplus_{i = 1}^n \left(M_{i, Z}^*\right)^{\oplus a_i}\right) \to \mathrm{Ext}^1(I_Z, \mathcal{O}_Z) \to 0;
\end{equation}

\begin{equation}
0 \to \mathrm{Tor}_1\left(\bigoplus_{i = 1}^n \left(M_{i}^*\right)^{\oplus a_i}, \mathcal{O}_Z\right) \to \mathcal{O}_Z^{\oplus a} \oplus \left( \bigoplus_{i = 1}^n M_{i, Z}^{\oplus a_i}\right) \to \mathcal{O}_Z^{\oplus 2r - a} \to  \bigoplus_{i = 1}^n \left(M_{i, Z}^*\right)^{\oplus a_i} \to 0,
\end{equation}
where $M_{i, Z} \coloneqq M_i \otimes_{\mathcal{O}_X} \mathcal{O}_Z$, and $M_i^*$ is the $R$-dual of $M_i$ for each $i$. Note that $\mathrm{Syz}_1(I_Z) = \mathcal{O}_{X}^{\oplus a} \oplus \left(\bigoplus_{i = 1}^nM_i^{\oplus a_i}\right)$, then $\mathrm{Syz}_2(I_Z) = \bigoplus_{i = 1}^n\left(M_{i}^*\right)^{\oplus a_i}$. Then the above two sequences combined give:
\begin{equation}\label{ext}
\mathrm{Ext}^1(I_Z, \mathcal{O}_Z) = \dfrac{\dfrac{\mathcal{O}_Z^{\oplus a} \oplus \left( \bigoplus_{i = 1}^n M_{i, Z}^{\oplus a_i}\right)}{\mathrm{Tor}_1\left(\bigoplus_{i = 1}^n \left(M_{i}^*\right)^{\oplus a_i}, \mathcal{O}_Z\right)}}{\dfrac{\mathcal{O}_Z^{\oplus r + 1}}{\mathrm{Hom}(I_Z, \mathcal{O}_Z)}}.
\end{equation}
We write $s, s_i, t_i$, resp. for the lengths of the $\mathcal{O}_Z$-modules $\mathrm{Hom}(I_Z, \mathcal{O}_Z), M_{i, Z}, \mathrm{Tor}_1(M_{i}, \mathcal{O}_Z)$, resp. (or, their dimensions as $\mathbb{C}$-vector spaces). The dimensions $t_i, s_i$ can be computed from the injective periodic resolution of the $M_i$'s. Given the periodic minimal projective resolution of $M_{i}^*$, a minimal injective resolution of $M_{i}$ is obtained by duality:
\begin{equation}\label{injresol}
0 \to M_i  \to \mathcal{O}_X^{\oplus 2\mathrm{rk}(M_i)} \xrightarrow{D_0} \mathcal{O}_X^{\oplus 2\mathrm{rk}(M_i)} \xrightarrow{D_1} \mathcal{O}_X^{\oplus 2\mathrm{rk}(M_i)} \to \dots
\end{equation}
Tensoring (\ref{injresol}) with $\mathcal{O}_Z$ we obtain a periodic complex of $\mathcal{O}_Z$-modules:
\begin{equation}
0 \to M_{i, Z} \to \mathcal{O}_Z^{\oplus 2\mathrm{rk}(M_i)} \xrightarrow{d_0} \mathcal{O}_Z^{\oplus 2\mathrm{rk}(M_i)} \xrightarrow{d_1} \mathcal{O}_Z^{\oplus 2\mathrm{rk}(M_i)} \to \dots
\end{equation}
In particular, $M_{i, Z} = \ker(d_0)$, and $\mathrm{Tor}_1(M_i, \mathcal{O}_Z) = \ker(d_1)/\coker(d_0)$. By symmetry and periodicity we also have $M_{i, Z}^* = \ker(d_1)$ and $\mathrm{Tor}_1(M_{i, Z}^*, \mathcal{O}_Z) = \ker(d_1)/\coker(d_0)$. In particular, $\mathrm{tor}_1(M_{i, Z}^*, \mathcal{O}_Z) = \mathrm{tor}_1(M_{i, Z}, \mathcal{O}_Z)$ for $i = 1, \dots, n$. Therefore, counting the dimensions in (\ref{ext}) we have
\begin{equation}
s + \sum_{i = 1}^n a_i(s_i - t_i - d) = 2d.
\end{equation}
Hence the tangent space $\mathrm{Hom}(I_Z, \mathcal{O}_Z)$ has dimension $2d$ if and only if for each $i = 1, \dots, n$ either $a_i = 0$ or $s_i = d + t_i$ is satisfied. On the other hand, locally each $M_i$ is a direct summand in the regular representation of the finite group $G$ of the singularity, hence $s_i$ can be computed as follows: Extend the ideal $I_Z$ of $Z$ in $\mathcal{O}_X$ to an ideal $\tilde{I}_Z$ of $\mathbb{C}[x, y]$. Then $\tilde{I}_Z$ still has finite colength. We fix a collection of monomials of $\mathbb{C}[x, y]$ which span $V = \mathbb{C}[x, y]/\tilde{I}_Z$ as a complex vector space, inside of which there is the $d$-dimensional subspace $\mathcal{O}_Z$. Now a minimal generating set of $M_i$ as an $S$-sub-$R$-module consists of polynomials $f^{1, 2}_{j}$ in $S$ for $j = 1, \dots, \mathrm{rk}(M_i)$. Then $M_{i, Z}$ is the $\mathcal{O}_Z$-module obtained by the union of translation of $\mathcal{O}_Z$ in $V$ by $\{f^{1, 2}_{j}\}$. In particular, the dimension of $M_{i, Z}$ is $s_i = 2d - \lambda_i$ where $\lambda_i$ is the number of pairs of monomials $g_1, g_2$ as part of vector space basis of $\mathcal{O}_Z$ such that $f_j^1g_1 = f_j^2g_2$ for some $j$. Since $1 \in \mathcal{O}_Z$ cannot be in such a ``syzygetic'' pair, there are at most $d - 1$ pairs. Hence $s_i > d$. Therefore, for the dimension $s$ of $\mathrm{Hom}(I_Z, \mathcal{O}_Z)$ to be equal to $2d$ the only case is where $a_i = 0$ for all $i = 1, \dots, n$, which in turn means that $\mathrm{Syz}_1(I_Z)$ is free.
\end{proof}
\end{theorem}

Recall from \cite[Section 3.6]{H01} that $\mathrm{Hilb}^d(\mathbb{A}^2)$ has trivial canonical bundle: $\omega_{\mathrm{Hilb}^d(\mathbb{A}^2)} \cong \mathcal{O}_{\mathrm{Hilb}^d(\mathbb{A}^2)}$. A careful inspection of Haiman's proof leads to a generalization to the singular affine surface $X_n$ with a single $A_n$ singularity. We write $H_d = \mathrm{Hilb}^d(X_n)$ as the Hilbert scheme of $d$ points for simplicity of notation. 
\begin{theorem}\label{canonicalbundle}
The irreducible scheme $H_d$ has a nowhere vanishing $2d$-form.
\end{theorem}

\begin{lemma}[Residue formula]
The affine surface $X_n$ has a rational nowhere vanishing 2-form. In fact, the same statement is true for any hypersurface in $\mathbb{A}^3$.
\begin{proof}
Let $dx \wedge dy \wedge dz$ be the volume form on $\mathbb{A}^3$, and write $f = xz - y^{n + 1}$. Then 
\[
\dfrac{dx \wedge dy}{\dfrac{\partial f}{\partial z}} = - \dfrac{dy \wedge dz}{\dfrac{\partial f}{\partial x}} = \dfrac{dz \wedge dx}{\dfrac{\partial f}{\partial y}}  
\]
defines this nowhere vanishing 2-form on $X_n$. This is independent of $f$.
\end{proof}
\end{lemma}

We follow the notations of Haiman: for a linear form $l$ in $x, y, z$ write $U_l$ for the open subset of $H_d$ parameterizing subschemes $Z$ such that $\mathcal{O}_Z \cong \langle 1, l, \dots, l^{d - 1}\rangle$ as $d$-dimensional vector spaces, and by $W$ the open subset of $H_d$ parameterizing reduced length $d$ subschemes.

\begin{lemma}\label{openset}
The complement of $U_x \cup U_z$ has codimension at least 2 in $H_d$.
\begin{proof}
We prove the statement in the case of the quadric cone $Q$, and show that $U_x$ is open dense in $H_d$. The argument remains valid for general $X_n$. Note that a general point $[Z] \in U_x$ corresponds to an ideal $I$ of $R$ of the form: 
\[
I = \langle x^d - a_0 - a_1x - \dots - a_{d - 1}x^{d - 1}, y - b_0 - b_1x - \dots - b_{d - 1}x^{d - 1},  z - c_0 - c_1x - \dots - c_{d - 1}x^{d - 1} \rangle.
\]
Suppose $Z$ is a reduced scheme supported on the line $l_x = \langle y, z \rangle \subset Q$, then the coefficients $a_0, \dots, a_{d - 1}$ are the elementary symmetric functions on the $x$-coordinates of the $d$ points of $Z$. Then $b_0, \dots, b_{d - 1}$ are uniquely determined by the interpolation relation $y_i = \mathbf{b}(x_i)$ for the $y$-coordinates of these points (here $\mathbf{b}(x_i) = b_0 + b_1x_i + \dots + b_{d - 1}x_i^{d - 1}$). Once the $a_i$ and $b_i$ are determined and $a_0 \neq 0$, then the $c_i$ are solved by the syzygetic relation given by $f = xz - y^2$. If $a_0 = 0$, then we see that $b_0 = 0$ as well. Consequently $c_0, \dots, c_{d - 2}$ are uniquely solved in terms of $a_1, \dots, a_{d -1}, b_1, \dots, b_{d - 1}, c_{d - 1}$, and $c_{d - 1}$ will not appear in the eventual expression of $z$ in terms of $x, \dots, x^{d - 1}$. This shows that $U_x$ is an open subset of $\mathbb{A}^{2d}: U_x = \mathrm{Spec}(\mathbb{C}[a_0, \dots, a_{d - 1}, b_0, \dots, b_{d - 1}]_{a_0}) \cup \mathrm{Spec}(\mathbb{C}[a_1, \dots, a_{d - 1}, b_1, \dots, b_{d - 1}])$.

The complement of $U_x \cup U_z$ in $W$ consists of schemes consisting of $d$ distinct points $p_1, \dots, p_d \in Q$ with at least two indices $i \neq j$ and two indices $k \neq l$ such that $x(p_i) = x(p_j)$ and $z(p_k) = z(p_l)$. This is a codimension 2 subset of $W$. The complement $H_d \setminus W$ has codimension 1 in $H_d$. Now a general point in $H_d \setminus W$ represents a scheme $Z$ with a length $2$ subscheme $Z_1$ and $d - 2$ reduced closed points, which constitute an open subset of $H_d \setminus W$. This open subset is not contained in $H_d \setminus (U_x \cup U_z)$. 
\end{proof}
\end{lemma} 

\begin{proof}[Proof of Theorem \ref{canonicalbundle}]
We first prove the statement over the open dense subset $W$. Restricting to $W \cap U_x$, by the proof of Lemma \ref{openset} we note that the coefficients $a_0, \dots, a_{d - 1}, b_0, \dots, b_{d - 1}$ are regular functions on $U_x$. The rational $2d$-form $dx_1 \wedge \dots \wedge dx_d \wedge dy_1 \wedge \dots \wedge dy_d / \Pi_{i = 1}^d(\partial f / \partial z_i)$ on $X_n^{d}$ is $S_d$-invariant and hence descends to a regular $2d$-form on the smooth locus of $X_n^{(d)}$, which in turn lifts to a regular $2d$-form to $W \cap U_x$. In particular, if we write $W^{\circ}$ for the open locus in $W$ of reduced schemes supported off the singularity in $X_n$, then we have the identity over $W^{\circ} \cap U_x$:
\begin{equation}\label{2dform}
\dfrac{d a_0 \wedge d a_1 \wedge \dots \wedge da_{d - 1} \wedge db_0 \wedge \dots \wedge db_{d - 1}}{a_0} = (-1)^{d} \dfrac{dx_1 \wedge \dots \wedge dx_d \wedge dy_1 \wedge \dots \wedge dy_d}{\Pi_{i = 1}^d(\partial f / \partial z_i)}.
\end{equation}
Next we notice that $W \setminus W^{\circ}$ parameterizes reduced schemes of the form $Z = z_1 +  \dots + z_d$ where $z_1 = 0$. This has codimension 2 in $W$, hence the rational form \ref{2dform} extends to a regular $2d$-form on $W \cap U_x$. By the symmetry of $x$ and $z$ in the defining equation of the surface $X_n$, we obtain another regular $2d$-form on $W^{\circ} \cap U_z$ analogous to \ref{2dform}. 

Explicitly, a point in $U_z$ represents a scheme that is defined by an ideal of the form 
\[ I = \langle  z^d - A_0 - A_1z - \dots - A_{d - 1}z^{d - 1}, y - B_0 - B_1z - \dots - B_{d - 1}z^{d - 1}, x - C_0 - C_1z - \dots - C_{d - 1}z^{d - 1}\rangle.
\]
The $2d$-forms $d a_0 \wedge d a_1 \wedge \dots \wedge da_{d - 1} \wedge db_0 \wedge \dots \wedge db_{d - 1}/a_0$ and $d A_0 \wedge d A_1 \wedge \dots \wedge dA_{d - 1} \wedge dB_0 \wedge \dots \wedge dB_{d - 1}/A_0$ coincide over the intersection $W \cap U_x \cap U_z$. So there is a nowhere vanishing $2d$ form on $U_x \cup U_z$ which extends to a nowhere vanishing $2d$ form on $H_d$.
\end{proof}

The rest of the section is inspired by \cite{H98} and \cite{ES14}. Denote by $A = \mathbb{C}\llbracket u^2, uv, v^2 \rrbracket$ and $R = \mathbb{C}[u^2, uv, v^2]$ the complete local ring of the affine quadric surface $Q \subset \mathbb{A}^3$ at the vertex $0 \in Q$ and the affine coordinate ring of $Q$ respectively. Fix a positive integer $d$, and write $B = B^0$ for the coordinate ring of $Q^{(d)}$. The $2d$-variable polynomial ring $R_{d} = \mathbb{C}[u_1, \dots, u_d, v_1, \dots, v_d]$ is a $B$-module. In $R_{d}$ there is a sub-$B$-module $B^1$, consisting of $S_d$-alternating polynomials, i.e.,
\[
B^1 = \{f \in R_d \mid \deg(f) \equiv 0(2),  \sigma  (f) = \mathrm{sgn}(\sigma) \cdot f, \forall \sigma \in S_d\}.
\]
We define $B^r \coloneqq \{f_1 \cdot \dots \cdot f_r \mid f_i \in B^1, \forall i = 1, \dots, r\}$ for any $r \geq 2$, and $\mathbf{B} = \bigoplus_{i \geq 0}B^i$.
Note that for any elements $f \in B^r$ and $g \in B^s$, we have $fg \in B^{r + s}$. Hence $\mathbf{B}$ is a graded $B^0$-algebra.

Suppose $D = \{(p_1, q_1), \dots, (p_d, q_d)\} \subset \Lambda$ is any subset with cardinality $\abs{D} = d$. Associated to $D$ there is the monomial matrix $M(D) \coloneqq \begin{bmatrix} u^{p_i}v^{q_i}\end{bmatrix}$ up to a permutation of the elements of $D$. We define the \textit{lattice discriminant} associated to $D$ to be the discriminant $\det(D) \coloneqq \det M(D)$. Up to sign, $\det(D)$ is a well-defined $S_d$-alternating polynomial, hence it defines an element in $B^1$. 

Write $\Lambda \subset \mathbb{N}^2$ for the sub-lattice of the semigroup of $R$. The minimal generating set of $\Lambda$ is denoted by $\Lambda_0 = \{(2, 0), (1, 1), (0, 2)\}$. A \textit{staircase} in $\Lambda$ is a subset $\Delta \subset \Lambda$ such that its complement $E(\Delta) \coloneqq \Lambda \setminus \Delta \subset \Lambda$ is closed under the semi-group addition, i.e., $E(\Delta) + \Lambda_0 = E(\Delta)$. Let $\Delta \subset \Lambda$ be a staircase, the \textit{corner set} of $\Delta$ is the subset $\mathcal{C}(\Delta) \subset E(\Delta)$ that minimally generate $E(\Delta)$ under the semi-group addition, i.e., $\mathcal{C}(\Delta) = \{\alpha \in E(\Delta) \mid \alpha - \mathbf{a} \in \Delta,  \forall \mathbf{a} \in \Lambda \}$. Let $\Delta$ be a staircase, the \textit{border} of $\Delta$ is the set of vertices $\mathcal{B}(\Delta) = (\bigcup_{\mathbf{a} \in \Lambda_0} \langle \Delta + \mathbf{a} \rangle) \setminus \Delta$.

The following theorem is a generalization of \cite[Proposition 2.6]{H98}.

\begin{theorem}\label{blowup}
The scheme $\mathrm{Proj}(\mathbf{B})$ with the natural projection $\theta: \mathrm{Proj}(\mathbf{B}) \to Q^{(d)} = \mathrm{Spec}(B)$ is isomorphic to $\mathrm{Hilb}^d(Q)$ with the Hilbert-Chow morphism over $Q^{(d)}$.
\end{theorem}

To prove the theorem we give an open covering of $\mathrm{Hilb}^d(Q)$ and a system of coordinates on each open set. Fixing a staircase $\Delta$ of cardinality $d$, we define a subfunctor $\underline{Hilb}^d_{Q, \Delta}: \mathbb{C}-\underline{\mathrm{Sch}} \To \underline{\mathrm{Set}}$ of the Hilbert functor of points as follows: For a scheme $U$, $\underline{Hilb}^d_{Q, \Delta}(U)$ is the set of closed subschemes $Y \subset Q \times U$, flat and finite of degree $d$ over $U$, such that the composition of $\mathcal{O}_U$-algebra homomorphisms
$\phi_{\Delta, U}(Y): \mathcal{O}_U[\Delta] \xhookrightarrow{\iota} \mathcal{O}_U\otimes_{\mathbb{C}} R \rightarrow \mathcal{O}_Y$ is surjective, where $\mathcal{O}_U[\Delta]$ denotes the rank $d$ locally free $\mathcal{O}_U$-submodule of $\mathcal{O}_U\otimes_{\mathbb{C}} R$ generated by elements whose exponents are in $\Delta$ with $\mathcal{O}_U$-algebra structure induced by the inclusion $\iota$ ($\mathbb{C}[\Delta]$ is the $d$-dimensional $\mathbb{C}$-vector space spanned by $u^av^b$ for $(a, b) \in \Delta$ and as a $\mathbb{C}$-algebra $\mathbb{C}[\Delta] \cong R/I_{\Delta}$).

\begin{prop}
This subfunctor $\underline{Hilb}^d_{Q, \Delta}$ is represented by a subscheme $\mathrm{Hilb}^d_{\Delta}(Q) \subset \mathrm{Hilb}^d(Q)$. Any closed point $[Z]$ of $\mathrm{Hilb}^d_{\Delta}(Q)$ can be generalized to $[Z_{\Delta}]$, defined by the initial ideal $I_{\Delta}$ of $I_Z$ via some 1-parameter subtorus $\mathbb{G}_m \subset \mathbb{T}$.
\begin{proof}
Let $U$ be any scheme, and let $f: U \To \mathrm{Hilb}^d(Q)$ be a morphism which induces a family of closed subschemes $Y$ of $U \times Q$, flat and finite of degree $d$ over $U$. Suppose $\phi_{\Delta, U}(Y)$ is surjective. Suppose $x \in U$ is a (not necessarily closed) point in $U$ with coordinate ring $\mathcal{O}_x = \mathcal{O}_U/\mathfrak{m}_x$. Write $Y_x$ as the fiber of $Y$ over $x$, and $I_x$ as the ideal of $Y_x$ in $R \otimes_{\mathbb{C}} \mathcal{O}_x$. The composition $\phi_{\Delta, U}(Y)$ restricts to an isomorphism of finite $\mathcal{O}_x$-algebras: $\phi_x: \mathcal{O}_x[\Delta] \To \mathcal{O}_{Y_{x}}$. Hence there exists $g_1, \dots, g_r \in \mathcal{O}_x[\Delta]$ such that
\[
\{f_i = u^{\alpha_i}v^{\beta_i} - g_i \mid i = 1, \dots, r, (\alpha_i, \beta_i) \in \mathcal{C}(\Delta), \mathrm{Exp}(g_i) \in \Delta\}
\]
is a generating set of the ideal $I_x$. Since $\Delta$ is a finite set, one can choose a weight pair $(\lambda_1, \lambda_2)$ for a one-parameter torus so that the weights of $u^{\alpha_i}v^{\beta_i}$ are all higher than the weight of any monomials with exponents inside $\Delta$. As a result, $\lim_{t \to 0}(t^{\lambda_1}, t^{\lambda_2})\cdot f_i = u^{\alpha_i}v^{\beta_i}$ for all $i$. Comparing the length we have $\langle u^{\alpha_1}v^{\beta_1}, \dots, u^{\alpha_r}v^{\beta_r}\rangle = I_{\Delta}$. This shows that $f(x) \in \mathrm{Hilb}^d_{\Delta}(Q)$.

Conversely, for any point $i: x \hookrightarrow \mathrm{Hilb}^d_{\Delta}(Q)$ corresponding to a closed subscheme $Z$ of $Q \times x$, by definition of the subset $\underline{Hilb}^d_{Q, \Delta}(x)$ there exists a set of generators of $I_Z$ of the form $f_i = u^{\alpha_i}v^{\beta_i} - g_i$ for $i = 1, \dots, r$. Choosing the one-parameter torus as above, one shows that $\phi_x: \mathcal{O}_x[\Delta] \to \mathcal{O}_Z$ is an $\mathcal{O}_x$-linear isomorphism.

The surjectivity of $\phi_{\Delta, U}$ is an open condition, so the subscheme $\mathrm{Hilb}^d_{\Delta}(Q)$ is open in $\mathrm{Hilb}^d(Q)$. 
\end{proof}
\end{prop}

Next we describe the affine coordinate ring of each $\mathrm{Hilb}^d_{\Delta}(Q)$. Suppose $\Delta$ is a staircase of cardinality $d$ with corner set $\mathcal{C}(\Delta) = \{(\alpha_0, \beta_0), \dots, (\alpha_r, \beta_r)\}$. Let $[Z] \in \mathrm{Hilb}^d_{\Delta}(Q)$ be a closed points.

\begin{lemma}\label{coordinate system}
There is a generating set of the ideal defining the scheme $Z$ of the form $\{f_0, \dots, f_r\}$ where $f_i = u^{\alpha_i}v^{\beta_i} - g_i$ with $g_i = \sum_{(a, b) \in \Delta}C_{a, b}^{\alpha_i, \beta_i}u^av^b$.
\begin{proof}
By the definition of $\mathrm{Hilb}^d_{\Delta}(Q)$, fixing the scheme $Z$ with ideal $I_Z \subset R$, any monomial in $R$ is equivalent to a polynomial whose exponent sets is a subset of $\Delta$ modulo the ideal $I_Z$, i.e., for any $(r, s) \in \Lambda$, there exist $C^{r, s}_{a, b} \in \mathbb{C}$ such that
\[
u^rv^s \equiv \sum_{(a, b) \in \Delta}C^{r, s}_{a, b}u^av^b (\textrm{mod } I_Z).
\]
The relations among the $C^{r, s}_{a, b}$ are give as follows. Multiplying both sides of the above congruence equation by $u^2, uv$ or $v^2$ we obtain
\begin{equation*}
C^{r + 2, s}_{a, b} = \sum_{(a', b') \in \Delta} C^{r, s}_{a', b'}C^{a' + 2, b'}_{a, b}, \quad C^{r + 1, s + 1}_{a, b} = \sum_{(a', b') \in \Delta} C^{r, s}_{a', b'}C^{a' + 1, b' + 1}_{a, b},
\quad C^{r, s+ 2}_{a, b} = \sum_{(a', b') \in \Delta} C^{r, s}_{a', b'}C^{a', b' + 2}_{a, b}.
\end{equation*}
Hence for any ideal $I_Z$ one can use these relations to reduce the generating set such that each generator has exactly one term that is not in $\Delta$ but in $\mathcal{C}(\Delta)$.
\end{proof}
\end{lemma}

As $Z$ varies the $C^{r, s}_{a, b}$'s define a collection of regular functions on $\mathrm{Hilb}^d_{\Delta}(Q)$. The relations in the proof of the previous lemma hold valid as relations among regular functions on $\mathrm{Hilb}^d_{\Delta}(Q)$.

\begin{prop}\label{cab}
Let $\Delta$ be a staircase in $\Lambda$ of cardinality $d$. The affine coordinate ring of the scheme $\mathrm{Hilb}^d_{\Delta}(Q) = \mathrm{Spec}(S^{\Delta})$ is given by $S^{\Delta} = \mathbb{C}[C^{\alpha}_{\beta}]/I^{\Delta}$ for $\alpha \in \mathcal{B}(\Delta)$, $\beta \in \Delta$ and
\begin{equation}
I^{\Delta} = \left\langle C^{\alpha + \lambda}_{\beta} - \sum_{\gamma \in \Delta}C^{\alpha}_{\gamma}C^{\gamma + \lambda}_{\beta},  \sum_{\gamma \in \Delta}C^{\epsilon + \lambda}_{\gamma}C^{\gamma + \lambda'}_{\beta} - \sum_{\gamma \in \Delta}C^{\epsilon + \lambda'}_{\gamma}C^{\gamma + \lambda}_{\beta} \right\rangle,
\end{equation}
for all $\alpha \in \mathcal{B}(\Delta), \lambda, \lambda' \in \Lambda_0$ such that $\alpha + \lambda \in \mathcal{B}(\Delta)$, with the additional convention that $C^{\alpha}_{\beta} = 0$ and $C^{\beta}_{\beta} = 1$ for $\alpha, \beta \in \Delta, \alpha \neq \beta$.
\begin{proof}
We first show that $\mathrm{Hilb}^d_{\Delta}(Q)$ embeds into an infinite dimensional affine space $\mathrm{Spec}(S')$, where $S' = \mathbb{C}[C^{r, s}_{a, b}]$ for $(r, s) \in E(\Delta)$ and $(a, b) \in \Delta$ and the ideal defining $\mathrm{Hilb}^d_{\Delta}(Q)$ in $\mathrm{Spec}(S')$ is of the same form as $I_{\Delta}$ as in the statement of the lemma without the restriction that $\alpha$ is chosen to be inside $\mathcal{B}(\Delta)$ such that $\alpha + \lambda \in \mathcal{B}(\Delta)$. We denote this ideal by $I'_{\Delta}$ and $S'_{\Delta} = S'/I'_{\Delta}$.

Let $Y$ be a family of length $d$ subscheme of $Q$ over a scheme $U$ with the property that $\rho: \mathcal{O}_U[\Delta] \xhookrightarrow{\iota} \mathcal{O}_U \otimes_k S \xrightarrow{\mu} \mathcal{O}_Y$ is surjective, and let $\phi: U \to \mathrm{Hilb}^d_{\Delta}(Q)$ be the corresponding morphism. The coordinate ring of $U$ is denoted by $B = \Gamma(U, \mathcal{O}_U)$ and the ideal of $Y$ in $U \times Q$ is denoted by $I_Y$. We define a $\mathbb{C}$-algebra homomorphism $\tau: S'_{\Delta} \to B$ as follows. For any $(r, s) \in E(\Delta)$,  since $\rho$ is surjective $\mu(1 \otimes u^rv^s) \in B \otimes_{\mathbb{C}} S /I_Y$ lifts to a unique element in $B[\Delta]$ as an $B$-linear combination of $u^av^b$'s for $(a, b) \in \Delta$. We define $\tau(C^{r, s}_{a, b})$ to be the coefficient of $u^av^b$ of the image of $u^rv^s$ under the lifting map.

Now we check that this is a well-defined $\mathbb{C}$-algebra homomorphism. Since $\mu$ is the canonical $\mathcal{O}_U$-algebra surjection, by comparing $\mu(1 \otimes u^rv^s)$ with $\mu(1 \otimes u^{r + 2}v^s)$ (and $\mu(1 \otimes u^{r + 1}v^{s + 1})$, $\mu(1 \otimes u^{r}v^{s + 2})$) we see that $\tau (C^{\alpha + \lambda}_{\beta} - \sum_{\gamma \in \Delta}C^{\alpha}_{\gamma}C^{\gamma + \lambda}_{\beta}) = 0$ for any vertex $\alpha \in E(\Delta)$ and basis element $\lambda \in \Lambda_0$. This shows that $\tau$ is multiplicative. Also, we consider any $(r_i, s_i)$ for $i = 1, 2, 3, 4$ such that the four vertices in $\Lambda$ form a parallelogram. By induction on the lengths of the sides of the parallelogram, it is enough to consider those parallelograms whose interior is a fundamental domain of the lattice, i.e., those which do not contain any proper subparallelograms. So we assume that $(r_i, s_i) - (r_1, s_1) = \lambda_i$ for $i = 2, 3$ and $\lambda_i \in \Lambda_0$. By the identity
\[
\mu(1 \otimes u^{r_4}v^{s_4}) = \mu(1 \otimes u^{r_3}v^{s_3}) \mu(1 \otimes (uv)^{\lambda_2}) =  \mu(1 \otimes u^{r_2}v^{s_2}) \mu(1 \otimes (uv)^{\lambda_1}),
\]
we see that
 \[
 \tau\left(\sum_{\gamma \in \Delta}C^{\epsilon + \lambda}_{\gamma}C^{\gamma + \lambda'}_{\beta} - \sum_{\gamma \in \Delta}C^{\epsilon + \lambda'}_{\gamma}C^{\gamma + \lambda}_{\beta}\right) = 0.
 \]

This shows that $\tau$ is a well-defined $\mathbb{C}$-algebra homomorphism. In particular, in the case when $U = \mathrm{Hilb}^d_{\Delta}(Q)$ there is a morphism $\mathrm{Hilb}^d_{\Delta}(Q) \To \mathrm{Spec}(S'_{\Delta})$.

On the other hand, one checks in the same way that $\mathrm{Spec}(S'_{\Delta})$ parameterizes a family $\tilde{S'_{\Delta}}$ of zero-dimensional schemes of length $d$ in $Q$ with the property that $S'_{\Delta}[\Delta] \xhookrightarrow{\iota} S'_{\Delta} \otimes_{\mathbb{C}} S \rightarrow \mathcal{O}_{\tilde{S'_{\Delta}}}$ is surjective. Therefore there is a morphism $g: \mathrm{Spec}(S'_{\Delta}) \To \mathrm{Hilb}^d_{\Delta}(Q)$. One checks that for any triple $(U, \phi: U \To \mathrm{Hilb}^d_{\Delta}(Q)$ and $\tau: S'_{\Delta} \To \Gamma(U, \mathcal{O}_U))$ as above, the map $\phi$ factors as $\phi = g \circ \tau^*$.

We further reduce the dimension of the ambient affine space in the previous step. By repeatedly using the first collection of relations above, we can further reduce to the case where $(r_1, s_1) \in \Delta$.
\end{proof}
\end{prop}

In particular, the proof of the preceding proposition shows the following
\begin{cor}
Notations as before, fixing a staircase $\Delta$, the $C^{\alpha}_{\beta}$ form a system of parameters at the $\mathbb{T}$-fixed point $Z_{\Delta}$ with monomial ideal $I_{\Delta}$. The maximal ideal $\mathfrak{m}_{\Delta} \coloneqq\langle C^{\alpha}_{\beta}\rangle$ defines this point $Z_{\Delta}$ as the origin of the affine space $\mathrm{Spec}(S^{\Delta}) = \mathrm{Spec}(\mathbb{C}[C^{\alpha}_{\beta}])$ via the embedding given in the previous proposition.
\end{cor}

\begin{cor}
Write $\widetilde{\mathrm{Hilb}}^d_{\Delta}(Q) \coloneqq \pi_d^{-1}(\mathrm{Hilb}^d_{\Delta}(Q))$ for the preimage in the universal family. Then as a (not necessarily reduced) subscheme of $\mathrm{Hilb}^d_{\Delta}(Q) \times Q$, $\widetilde{\mathrm{Hilb}}^d_{\Delta}(Q)$ is defined by $\tilde{\mathcal{I}_{\Delta}} = \langle u^rv^s - \sum_{(a, b) \in \Delta}C^{r, s}_{a, b}u^av^b \rangle$, where, by abuse of notation, $C^{r, s}_{a, b}$ are the regular functions on $\mathrm{Hilb}^d_{\Delta}(Q) \times Q$ that are pulled back from $\mathrm{Hilb}^d_{\Delta}(Q)$ and $u, v$ are the pullbacks from $Q$.
\begin{proof}
The proof of Proposition 2.9 in \cite{H98} works verbatim.
\end{proof}
\end{cor}

To prove Theorem \ref{blowup}, we work on each open subset $\mathrm{Hilb}^d_{\Delta}(Q)$ associated with a fixed staircase. We note that any $S_d$-alternating polynomial in $R_d$ is a linear combination of lattice discriminants, i.e., $B^1$ is $\mathbb{C}$-linearly spanned by the set of all lattice discriminants associated to subsets of $\Lambda$ of cardinality $d$. 

\begin{proof}[Proof of Theorem \ref{blowup}]
The proof of the theorem is almost identical to that of \cite[Proposition 2.6]{H98}. Suppose $\Delta = \{(a_i, b_i)\}_{i = 1, \dots, d}$ is a staircase and $D = \{(r_i, s_i)\}_{i = 1, \dots, d}$ is any subset of $\Lambda$ of cardinality $d$. We first work in the locus $U_{\Delta}$ of reduced schemes in $\mathrm{Hilb}^d_{\Delta}(Q)$, which is dense open since $\mathrm{Hilb}^d_{\Delta}(Q)$ is irreducible. For any point $[Z] \in U_{\Delta}$ the image $h([Z])$ is a 0-cycle of $d$ distinct points, hence $\det^2(\Delta)(h([Z])) \neq 0$ (we are taking the square of the lattice determinant since the determinant only defines an alternating function). 

By Lemma \ref{coordinate system} the lattice discriminants $\det(D)$ and $\det(\Delta)$ are related by
\[
h^*\left(\dfrac{\det(D) \det(\Delta)}{\det^2(\Delta)}\right) = \det\left(\begin{bmatrix}C^{r_i, s_i}_{a_j, b_j}\end{bmatrix}_{1 \leq i, j \leq d}\right),
\]
where $C^{r_i, s_i}_{a_j, b_j}$ are regular functions on $U_{\Delta}$. This implies that $h^*\left(\dfrac{\det(D) \det(\Delta)}{\det^2(\Delta)}\right)$ extends to a regular function on $\mathrm{Hilb}^d_{\Delta}(Q)$. Hence $h^*(B^2)$ is a principal ideal generated by $h^*(\det^2(\Delta))$ over $\mathrm{Hilb}^d_{\Delta}(Q)$. 

Note that $\mathrm{Proj}(\mathbf{B}) \cong \mathrm{Proj}(\mathbf{B}^{[2]})$, where $\mathbf{B}^{[2]} \coloneqq \bigoplus_{k = 2i, i \in \mathbb{N}}B^{k}$. By considering all the staircases $\Delta$ of cardinality $d$, the preceding paragraph shows that $\theta: \mathrm{Proj}(\mathbf{B}) \to Q^{(d)}$ is identified as the blow-up of $Q^{(d)}$ along the ideal $B^2 \subset B^0$. 

Now we use the universal property of blow-up to see that there is a projective morphism $f_{\Delta}: \mathrm{Hilb}^d_{\Delta}(Q) \to \mathrm{Proj}(\mathbf{B}^{[2]})|_{\theta^{-1}(h(\mathrm{Hilb}^d_{\Delta}(Q)))}$ over $h(\mathrm{Hilb}^d_{\Delta}(Q))$, which glues to a projective morphism $f: \mathrm{Hilb}^d(Q) \to \mathrm{Proj}(\mathbf{B}^{[2]})$. In fact, $f$ is birational and surjective since both sides of $f$ are irreducible, birational to $Q^{(d)}$, and projective over $Q^{(d)}$. 

To show that $f$ is in fact an isomorphism over $Q^{(d)}$ it remains to show that $f$ is an embedding. For that matter, it suffices to show that $f^*\mathcal{O}_{\mathrm{Proj}(\mathbf{B}^{[2]})} \to \mathcal{O}_{\mathrm{Hilb}^d(Q)}$ is surjective. Again we work locally over each open set $\mathrm{Hilb}^d_{\Delta}(Q)$. By Proposition \ref{cab} the system of functions $C_{a, b}^{r, s}$ generate the affine coordinate ring of $\mathrm{Hilb}^d_{\Delta}(Q)$ with $(a, b) \in \Delta$ and $(r, s) \in \Lambda \setminus \Delta$. Hence to show the surjectivity of $f^*$, it suffices to show that each $C_{a, b}^{r, s}$ is in the image of $f^*$. We take the subset $D$ of $d$ elements in $\Lambda$ by $D \coloneqq \Delta \setminus \{(a, b)\} \cup \{(r, s)\}$. Over $U_{\Delta}$ we have that $C_{a, b}^{r, s} = f^*\theta^*(\det(D)/\det(\Delta))$. This shows that $f$ is an embedding restricted to $U_{\Delta}$ and so as restricted to the closure of $U_{\Delta}$. Now we glue over all the staircases $\Delta$ to complete the proof.
\end{proof}

In light of the work of Haiman, we also mention some further geometric consequence of Theorem \ref{blowup}. Denote by $\mathcal{O}(1)$ the ample line bundle on $\mathrm{Hilb}^d(Q)$ associated with the blow-up structure given by Theorem \ref{blowup}. Note that there is the tautological bundle $\mathcal{T}_d \coloneqq (\pi_d)_*(\mathcal{O}_{\widetilde{\mathrm{Hilb}}^d(Q)})$ of rank $d$. 

\begin{cor}(cf. \cite[Proposition 2.12]{H98})
There is an identification of line bundles $\wedge^d \mathcal{T}_d \cong \mathcal{O}(1)$ on $\mathrm{Hilb}^d(Q)$.
\begin{proof}
Again the proof is identical to that of \cite[Proposition 2.12]{H98}, and we refer the reader to Haiman's paper for the details. 
\end{proof}
\end{cor}

\section{Examples on more general surface singularities}

As one might immediately ask, what happens for surface singularities other than rational double points? We provide some evidence supporting the speculation that rational double points constitute the only class of surface singularities that can have irreducible Hilbert schemes. First we fix some notations. Let $S = k\llbracket x_1, \dots, x_e\rrbracket$ be the power series ring in $e$ variables ($e \geq 4$) and $I_0$ be the ideal of $S$ defining a rational surface singularity of embedding dimension $e$ at the closed point $p$. Write $\mathfrak{M}$ for the maximal ideal of $S$, $\bar{S} \coloneqq \mathrm{gr}_{\mathfrak{M}}(S)$ the polynomial ring, $R \coloneqq S/I_0$, $X = \mathrm{Spec}(R)$ for the surface, and the maximal ideal of $R$ is $\mathfrak{m}$. Hence one can talk about the associated graded $R$-modules with respect to $\mathfrak{m}$. By Wahl and Reimenschneider \cite{R74, W77}, there is a minimal set of generators $\{F_1, \dots, F_{e - 1} \}$ of $I_0$ such that $\{in(F_1), \dots, in(F_{e - 1})\}$ is a collection of quadrics subjecting to linear relations (in fact, Wahl showed that the associated graded ideal $\mathrm{gr}_{\mathfrak{M}}(I_0)$ with standard filtration has a minimal homogeneous $\bar{S}$-free resolution such that all syzygies are linear).

\begin{eg}
Let $e = 4$, then any minimal graded free resolution of the tangent cone of such a surface singularity looks like the one for the affine cone over the rational normal curve:
\[
0 \to \bar{S}(-3)^{\oplus 2} \to \bar{S}(-2)^{\oplus 3} \to \bar{S} \to \bar{S}/I \to 0,
\]
where the two linear syzygies can be obtained by repeating one row of the $2 \times 3$ matrix whose maximal minors give the equations of the cone.
\end{eg}

Suppose $Z$ is a closed subscheme of $X$ of finite length with ideal $I_Z$ in $R$. Then the associated $\mathfrak{m}$-graded Artinian $k$-algebra of $\mathcal{O}_Z = R/I_Z$ is
\[
\mathrm{gr}_{\mathfrak{m}}(R/I_Z) = (\mathrm{gr}_{\mathfrak{m}}R)/in(I_Z),
\]
where $in(I_Z)$ is the initial ideal of $I_Z$ with respect to powers of $\mathfrak{m}$. This graded algebra defines a zero-dimensional subscheme of the tangent cone of $X$ of the same length as $Z$. By sending a zero-dimensional quotient of $R$ to the associated graded ring, one defines a morphism 
\[
\pi_d: \mathrm{Hilb}^d(X) \to H^d(C(X)).
\]
Here $H^d(C(X))$ is the standard graded Hilbert scheme on the tangent cone $C(X)$ (a special case of multigraded Hilbert scheme \cite{HS04}).

We recall the construction of a non-smoothable component in the Hilbert scheme of at least 8 points in an affine space of dimension at least 4 (\cite[Section 5]{8points}). Consider local $\mathbb{C}$-algebras of Hilbert function $h = (1, 4, 3)$ giving rise to length 8 subschemes of $\mathbb{A}^4$ supported at a single point. They show that at a general such point, the Zariski tangent space has dimension less than 32, the expected dimension of the principal component. The general point of their explicit choice is defined by the ideal of the polynomial ring in the $x_i$'s:
\[
I = \langle x_1^2, x_1x_2, x_2^2, x_3^2, x_3x_4, x_4^2, x_1x_4 + x_2x_3 \rangle,
\]
and in arbitrary higher dimensional $\mathbb{A}^d$ for $d > 4$ by adding more variables:
\[
I = \langle x_1^2, x_1x_2, x_2^2, x_3^2, x_3x_4, x_4^2, x_1x_4 + x_2x_3 \rangle + \langle x_i \mid 4 < i \leq d \rangle.
\]
Moreover, the authors describe the components in detail:
\begin{itemize}
\item[1.] $\mathrm{Hilb}^8(\mathbb{A}^4)$ has exactly 2 irreducible components, the principal component $R_8^4$ and the component $G_8^4$ which parameterizes schemes whose local Hilbert function is $(1, 4, 3)$.
\item[2.] The intersection $W = R_8^4 \cap G_8^4$ is an integral divisor in $G_8^4$, and the Zariski tangent space of $\mathrm{Hilb}^8(\mathbb{A}^4)$ at any point in $W$ is 33-dimensional.
\item[3.] We have $G_8^4 \cong G_0 \times \mathbb{A}^4$ where $G_0$ parameterizes length 8 local algebras supported at a fixed point. It contains the graded Hilbert scheme with Hilbert function $(1, 4, 3)$ by forgetting the grading. It is isomorphic to the Grassmannian $\mathrm{Gr}(7, S_2)$.
\item[4.] At each closed point of $G_8^4$ one can take a vector space basis $Q_1, Q_2, Q_3$ of the quadratic forms in the local algebra. Each $Q_i$ can be represented by a symmetric $4 \times 4$ matrix $A_i$ for $i = 1, 2, 3$. (This can also be done globally by taking the three sections of the universal bundle on the Grassmannian.) Then $W$ as a subscheme of $G_8^4$ is defined by a single equation, that is the Pfaffian of the $12 \times 12$ matrix:
\[
\begin{bmatrix}
    0 & A_1 & -A_2 \\
    -A_1 & 0 & A_3 \\
    A_2 & -A_3 & 0
  \end{bmatrix}.
\]
In Pl\"ucker coordinates, this Pfaffian is an irreducible quadric. The divisor $W$ of smoothable schemes in $G_8^4$ is linearly equivalent to $2H$, where $H$ is an effective generator of $\mathrm{Pic}(G_8^4)$. In particular, this statement distinguishes when a length 8 scheme $Z$ in $\mathbb{A}^4$ with local Artinian structure sheaf $\mathcal{O}_Z$ of Hilbert function $(1, 4, 3)$ is smoothable: take three generators of $\mathcal{O}_Z(2)$, then $Z$ is smoothable if and only if the corresponding Pfaffian is degenerated. 
\end{itemize}

We consider rational surface singularities $(R, p)$ of embedding dimension 4. Let $P_1, P_2, P_3$ be the leading quadratic forms of the three power series in 4 variables defining $(X = \mathrm{Spec}(R), p)$ in $\mathbb{A}^4$. 

\begin{prop}
If $P_1, P_2$ and $P_3$ are general, then the Hilbert scheme $\mathrm{Hilb}^d(X)$ is reducible for $d \geq 8$. 
\begin{proof}
For the singularity given by $P_1, P_2, P_3$, if $H^8(C(X))$ is reducible, then $\mathrm{Hilb}^8(X)$ is also reducible. To show that $H^8(C(X))$ is reducible, it suffices to find a length 8 scheme $Z$ with $\mathcal{O}_Z(2)$ defining a closed point in this cycle $\mathrm{Gr}(4, R_2)$. 

A length $8$ subscheme $Z$ in the tangent cone $C(X)$ supported at $p$ is defined by $V + \mathfrak{m}^3$, where $V$ is a four dimensional subspace of the second graded summand $R_2 = \langle P_1, P_2, P_3 \rangle^{\perp}$. Then one can form the Pfaffian $Pf(V^{\perp})$ as above with respect to any three quadratic forms spanning $\mathcal{O}_Z(2) = V^{\perp} \subset R_2$. Note that the Grassmannian $\mathrm{Gr}(4, R_2)$ as a codimension 9 cycle in $\mathrm{Gr}(7, S_2)$ parameterizing 7-planes containing a fixed 3-plane $\langle P_1, P_2, P_3 \rangle$ represents $\sigma_{3, 3, 3}$ in the standard Schubert calculus notations. 

If $P_1, P_2, P_3$ are chosen general, the cycle $\sigma_{3, 3, 3}$ will not be contained in the degeneration locus of the Pfaffian. Hence $H^8(C(X))$ is reducible, and in turn $\mathrm{Hilb}^8(X)$ is reducible. There exists some length $8$ subscheme of $X$ that is not smoothable. Therefore there exists a non-smoothable subscheme of $X$ of any length $d \geq 8$, and $\mathrm{Hilb}^d(X)$ is reducible. 
\end{proof}
\end{prop}

In the simplest case of the cone over a rational normal curve in $\mathbb{P}^3$ we can see the reducibility of the Hilbert scheme concretely. Take $\bar{S} = \mathbb{C}[x, y, z, w]$ and $I = \langle xz - y^2, xw - yz, yw - z^2 \rangle$ to be the ideal of the affine cone, $R = \bar{S}/I$, and $X = \mathrm{Spec}(R)$ to be the cone. The following construction is a simple modification from the example of \cite{8points}. 

\begin{prop}\label{twisted}
The Hilbert scheme $\mathrm{Hilb}^8(X)$ is reducible.
\begin{proof}
Define an ideal of $\bar{S}$:
\[
J_0 = \langle x^2, xy, xz - y^2, xw - yz, yw - z^2, zw, w^2 \rangle.
\]
In particular, note that the scheme $Z = \mathrm{Spec}(\bar{S}/J_0)$ still has length $8$ since $\langle x, y, z, w \rangle^3 \subset J_0$, and it is scheme-theoretically embedded in the singular surface $X$ since $I \subset J_0$.

Now we compute the Zariski tangent space of $\mathrm{Hilb}^8(\mathbb{A}^4)$ at $[Z]$: choose a $\mathbb{C}$-linear space basis of $\mathcal{O}_Z$ as $\{1, x, y, z, w, xw, xz, yw\}$. Then an $S$-linear homomorphism $\phi: J_0 \to S/J_0$ can be represented by a table of the form
\begin{center}
      \begin{tabular}[b]{ c | c c c c c c c}
    & $x^2$ & $xy$ & $xz - y^2$ & $xw - yz$ & $yw - z^2$ & $zw$ & $w^2$ \\ \hline
    1 &  0 & 0 & 0 & 0 & 0 & 0 & 0  \\
    $x$ & $b_2 + c_3$ & 0 & $c_1$ & 0 & 0 & $f_1$ & 0 \\
    $y$ & $a_2$ & $b_2$ & $-c_3$ & $c_1 - 2c_3$ & 0 & $f_2$ & $f_1$ \\ 
    $z$ & $a_3$ & $a_2$ & $c_3$ & 0 & $-2c_3$ & $f_3$ & $f_2$ \\
    $w$ & 0 & $a_3$ & 0 & $-c_3$ & $c_3$ & $c_1 - 2c_3$ & $f_3$ \\
    $xw$ & * & * & * & * & * & * & * \\
    $xz$  & * & * & * & * & * & * & * \\
    $yw$  & * & * & * & * & * & * & * \\
  \end{tabular}
\end{center}
where the *'s and $a_2, b_2, a_3, c_1, c_3, f_1, f_2, f_3$ in the table can take arbitrary values. The total number of such free entries in the table counts the dimension $\dim T_{[Z]}\mathrm{Hilb}^8(\mathbb{A}^4) = 21 + 8 = 29 < 4 \times 8 = 32$. We conclude that the closed point $[Z]$ does not lie on the main component of $\mathrm{Hilb}^8(\mathbb{A}^4)$ (not even the intersection of the main component with other components). But $[Z] \in \mathrm{Hilb}^8(X)$ as a closed point, and if $\mathrm{Hilb}^8(X)$ were irreducible then $\mathrm{Hilb}^8(X)$ is contained in the main component of $\mathrm{Hilb}^8(\mathbb{A}^4)$. This shows that $\mathrm{Hilb}^8(X)$ is reducible.
\end{proof}
\end{prop}

\bibliographystyle{amsalpha}

\begin{thebibliography}{CEVV}

\bibitem[AIK]{AIK}
A. Altman, A. Iarrobino, S. Kleiman. Irreducibility of the compactified Jacobian. In \textit{Real and Complex Singularities, Oslo 1976, Proc. 9th Nordic Summer School/NAVF}, P. Holm (ed.), Sijthoff and Noordhoff, (1977), pp.1-12.

\bibitem[AV85]{AV85}
M. Artin, J.-L. Verdier. Reflexive Modules Over Rational Double Points. \textit{Math. Ann.}, 270, 79-82, 1985.

\bibitem[A86]{A86}
M. Auslander. Rational singularities and almost split sequences. \textit{Trans. AMS.}, 293, 511-531, 1986.

\bibitem[B83]{B83}
A. Beauville. Vari\'et\'es k\"ahl\'eriennes dont la premi\`ere classe de Chern est nulle. \textit{J. Diff. Geom.}, 18, 755-782, 1983.





\bibitem[B77]{B77}
J. Brian\c{c}on. Description de $Hilb^n \mathbb{C}\{x, y\}$. \textit{Invent. Math.}, 41, 45-89, 1977.





\bibitem[CEVV]{8points}
D. Cartwright, D. Erman, M. Velasco, B. Viray. Hilbert schemes of 8 points. \textit{Algebra and Number Theory}, 3, 763-795, 2009.

\bibitem[deCM]{deCM}
M. de Cataldo, L. Migliorini. The Chow groups and the motive of the Hilbert scheme of points on a surface. \textit{J. Algebra}, 251, 824-848, 2002.




\bibitem[E80]{E80}
D. Eisenbud. Homological algebra on a complete intersection, with an application to group representations. \textit{Trans. AMS}, 260, 35-64, 1980.

\bibitem[ES14]{ES14}
T. Ekedahl, R. Skjelnes. Recovering the good component of the Hilbert scheme. \textit{Ann. Math.}, 179, 805-841, 2014.

\bibitem[EL99]{EL99}
G. Ellingsrud, M. Lehn. Irreducibility of the punctual quotient scheme of a surface. \textit{Arkiv f\"or Matematik}, 37, 245-254, 1999.

\bibitem[ES87]{ES87}
G. Ellingsrud, S. Str{\o}mme. On the homology of the Hilbert scheme of points in the plane. \textit{Invent. Math.}, 87, 343-352, 1987.

\bibitem[ES88]{ES88}
---------. On a cell decomposition of the Hilbert scheme of points in the plane. \textit{Invent. Math.}, 91, 365-370, 1988.

\bibitem[ES98]{ES98}
---------. An intersection number for the punctual Hilbert scheme of a surface. \textit{Trans. AMS}, 350, 2547-2552, 1998.


\bibitem[E12]{E12}
D. Erman. Murphy's law for Hilbert function strata in the Hilbert scheme of points. \textit{Matt. Res. Let.}, 19, 1277-1281, 2012.

\bibitem[E85]{E85}
H. Esnault. Reflexive modules on quotient surface singularities. \textit{J. reine angew. Math.}, 362, 63-71, 1985.

\bibitem[FGA]{FGA}
B. Fantechi, L. G\"ottsche, L. Illusie, et al. \textit{Fundamental Algebraic Geometry, Grothendieck's FGA Explained}, Mathematical Surveys and Monographs Vol. 123, AMS, (2005).

\bibitem[F68]{F68}
J. Fogarty. Algebraic familier on an algebraic surface. \textit{Amer. J. Math.}, 90, 511-521, 1968.

\bibitem[F83]{F83}
A. Fujiki. On primitively symplectic compact K\"ahler $V$-manifolds of dimension four. In \textit{Classification of algebraic and analytic manifolds (Katata, 1982)}, Progr. Math., Vol. 39, Birkh\"auser Boston, (1983), pp. 71-250.

\bibitem[G80]{G80}
P. Gabriel. Auslander-Reiten sequences and representation-finite algebras. In \textit{Representation theory I (Proc. Workshop Carleton Univ., Ottawa, Ont., 1979)}, Lecture Notes in Math., vol. 831, Springer, Berlin, 1980, pp. 1-71.



\bibitem[GV83]{GV83}
G. Gonzalez-Sprinberg, J.-L. Verdier. Construction g\'eom\'etrique de la correspondance de McKay. \textit{Ann. Sci. \'E. N. S. 4$^e$ s\'erie}, 16, 409-449, 1983.

\bibitem[G90]{G90}
L. G\"ottsche. The Betti numbers of the Hilbert scheme of points on a smooth projective surface. \textit{Math. Ann.}, 286, 193-207, 1990.

\bibitem[G98]{G98}
---------. A conjectural generating function for numbers of curves on surfaces. \textit{Comm. Math. Phys.}, 196, 523-533, 1998.

\bibitem[G88]{G88}
G. Gotzmann. A stratification of the Hilbert scheme of points in the projective plane. \textit{Math. Z.}, 199, 539-547, 1988.

\bibitem[G96]{G96}
I. Grojnowski. Instantons and affine algebras. I. The Hilbert scheme and vertex operators. \textit{Math. Res. Lett.}, 3, 275-291, 1996.

\bibitem[H98]{H98}
M. Haiman. $t, q$-Catalan numbers and the Hilbert scheme, \textit{Discrete Math.}, 193. 201-224, 1998.

\bibitem[H01]{H01}
---------. Vanishing theorems and character formulas for the Hilbert scheme of points in the plane. \textit{Invent. Math.}, 149, 371-407, 2001.

\bibitem[HS04]{HS04}
M. Haiman, B. Sturmfels. Multigraded Hilbert schemes. \textit{J. Alg. Geom.}, 13(4), 725-769, 2004.

\bibitem[HK]{HK}
J. Herzog, M. K\"{u}hl. Maximal Cohen-Macaulay modules over Gorenstein rings and Bourbaki-sequences. In \textit{Advanced Studies in Pure Mathematics, Commutative Algebra and Combinatorics}, vol. 11, 1987, pp. 65-92. 


\bibitem[I72]{I72}
A. Iarrobino. Reducibility of the families of 0-dimensional schemes on a variety. \textit{Invent. Math.}, 15, 72-77, 1972.


\bibitem[I00]{I00}
A. Ishii. Versal deformation of reflexive modules over rational double points. \textit{Math. Ann.}, 317, 239-262, 2000.

\bibitem[I05]{I05}
O. Iyama. $\tau$-Categories I: Ladders. \textit{Algebras and Representation Theory}, 8, 297-321, 2005.

\bibitem[IM]{IM}
O. Iyama, M. Wemyss. The classification of special Cohen-Macaulay modules. \textit{Math. Z.}, 265, 41-83, 2010.





\bibitem[MP13]{MP13}
R. Mir\'o-Roig, J. Pons-Llopis. Reducibility of punctual Hilbert schemes of cone varieties. \textit{Communications in Algebra}, 41, 1776-1780, 2013.


\bibitem[N97]{N97}
H. Nakajima. Heisenberg algebra and Hilbert schemes of points on projective surfaces. \textit{Ann. Math.}, 145, 379-388, 1997.

\bibitem[R80]{R80}
C. J. Rego. The compactified Jacobian. \textit{Ann. Sci. \'E. N. S.}, 13, 211-223, 1980.

\bibitem[R74]{R74}
O. Riemenschneider. Deformationen von Quotientensingularit\"aten \textit{Math. Ann.}, 209, (1974) 211-248.

\bibitem[R08]{Rydh08}
D. Rydh. Families of zero-cycles and divided powers: I. representability. arXiv:0803.0618v1.

\bibitem[S12]{S12}
V. Shende. Hilbert schemes of points on a locally planar curve and Severi strata of its versal deformation. \textit{Compositio Math.}, 148, 531-547, 2012.



\bibitem[T12]{T12}
Y-J. Tzeng. A proof of the G\"ottsche-Yau-Zaslow formula. \textit{J. Diff. Geom.}, 90, 439-472, 2012.

\bibitem[W77]{W77}
J. Wahl. Equations defining rational singularities. \textit{Ann. Sci. \'Ecole Norm. Sup.}, 10, 231-264, 1977.

\bibitem[W88]{W88}
J. Wunram. Reflexive Modules on Quotient Surface Singularities. \textit{Math. Ann.}, 279, 583-598, 1988.

\bibitem[Y04]{Y04}
Y. Yoshino. On degenerations of modules. \textit{J. Algebra}, 278, 217-226, 2004.

\end{thebibliography}

\end{document}